\documentclass[10pt, leqno]{article}
\usepackage[bottom]{footmisc}
\usepackage[margin=53mm, top=26mm, bottom=34mm, footskip=24pt]{geometry}
\usepackage{fancyhdr}
\makeatletter
\renewcommand{\maketitle}{\bgroup\setlength{\parindent}{0pt}
\begin{flushleft}
  \@author\\
  \@date
  \vspace{12pt}
  
  \textbf{\Large\@title}
\end{flushleft}\egroup
}
\renewenvironment{abstract}{
  \vspace{13pt}
  \begin{flushleft}%
      {\bfseries \abstractname\vspace{-.5em}}%\vspace{\z@}}%
  \end{flushleft}%
}
\makeatother
\usepackage{titlesec}
\titleformat{\section}
  {\normalfont\large\bfseries}{\thesection}{.5em}{}
\titlespacing\section{0pt}{12pt plus 4pt minus 2pt}{4pt plus 1pt minus 2pt}
\titleformat{\subsection}
  {\normalfont\normalsize\bfseries}{\thesubsection}{.4em}{}
\titlespacing\subsection{0pt}{6pt plus 2pt minus 1pt}{0pt plus 1pt minus 2pt}
\titleformat{\subsubsection}[runin]
  {\normalfont\normalsize\bfseries}{\thesubsubsection}{.3em}{\addperiod}
\titlespacing\subsubsection{0pt}{8pt plus 4pt minus 2pt}{4pt plus 1pt minus 2pt}
\newcommand{\addperiod}[1]{#1.}% ugly hack
\usepackage[page]{totalcount} % To have access to total page number
\usepackage{microtype} %Typographical secret sauce
\usepackage[parfill]{parskip} % Linebreak and no indent for paragraph
\usepackage{enumitem}
\usepackage[small,it,margin=2pt]{caption}

\usepackage[utf8]{inputenc}
\usepackage[T1]{fontenc}
\usepackage[greek,english]{babel} %greek
\let\LGRtextnu\textnu
\usepackage[scaled=0.96]{XCharter}
\let\textnu\LGRtextnu
\usepackage[scaled=1.03, varqu, varl]{inconsolata}
\usepackage[scaled=.95]{sourcesanspro}
\usepackage{amsthm}
\usepackage[xcharter, vvarbb, scaled=1.05]{newtxmath}
\usepackage[scr=boondoxo, frak=boondox, scrscaled=.97, bb=esstix, bbscaled=.96]{mathalfa}%[cal=euler, scr=boondoxo, frak=euler]{mathalfa}%

\makeatletter
\DeclareSymbolFont{timesoperators}{T1}{\sfdefault}{m}{n}
\def\operator@font{\mathgroup\symtimesoperators}
\makeatother

\usepackage{substitutefont}
\substitutefont{LGR}{\rmdefault}{artemisia}

\DeclareSymbolFont{greekletters}{LGR}{artemisia}{m}{it}
\SetSymbolFont{greekletters}{bold}{LGR}{artemisia}{b}{it}
\DeclareSymbolFont{ucgreekletters}{LGR}{artemisia}{m}{n}
\SetSymbolFont{ucgreekletters}{bold}{LGR}{artemisia}{b}{n}
\DeclareMathSymbol{\alpha}{\mathord}{greekletters}{`a}
\DeclareMathSymbol{\beta}{\mathord}{greekletters}{`b}
\DeclareMathSymbol{\gamma}{\mathord}{greekletters}{`g}
\DeclareMathSymbol{\delta}{\mathord}{greekletters}{`d}
\DeclareMathSymbol{\epsilon}{\mathord}{greekletters}{`e}
\DeclareMathSymbol{\zeta}{\mathord}{greekletters}{`z}
\DeclareMathSymbol{\eta}{\mathord}{greekletters}{`h}
\DeclareMathSymbol{\theta}{\mathord}{greekletters}{`j}
\DeclareMathSymbol{\iota}{\mathord}{greekletters}{`i}
\DeclareMathSymbol{\kappa}{\mathord}{greekletters}{`k}
\DeclareMathSymbol{\lambda}{\mathord}{greekletters}{`l}
\DeclareMathSymbol{\mu}{\mathord}{greekletters}{`m}
\DeclareMathSymbol{\nu}{\mathord}{greekletters}{`n}
\DeclareMathSymbol{\omicron}{\mathord}{greekletters}{`o}
\DeclareMathSymbol{\xi}{\mathord}{greekletters}{`x}
\DeclareMathSymbol{\pi}{\mathord}{greekletters}{`p}
\DeclareMathSymbol{\rho}{\mathord}{greekletters}{`r}
\DeclareMathSymbol{\sigma}{\mathord}{greekletters}{`s}
\DeclareMathSymbol{\tau}{\mathord}{greekletters}{`t}
\DeclareMathSymbol{\upsilon}{\mathord}{greekletters}{`u}
\DeclareMathSymbol{\phi}{\mathord}{greekletters}{`f}
\DeclareMathSymbol{\chi}{\mathord}{greekletters}{`q}
\DeclareMathSymbol{\psi}{\mathord}{greekletters}{`y}
\DeclareMathSymbol{\omega}{\mathord}{greekletters}{`w}
\DeclareMathSymbol{\varepsilon}{\mathord}{greekletters}{`e}
\DeclareMathSymbol{\vartheta}{\mathord}{greekletters}{`j}
\DeclareMathSymbol{\varpi}{\mathord}{greekletters}{`p}
\DeclareMathSymbol{\varrho}{\mathord}{greekletters}{`r}
\DeclareMathSymbol{\varsigma}{\mathord}{greekletters}{`v}
\DeclareMathSymbol{\varphi}{\mathord}{greekletters}{`f}
\DeclareMathSymbol{\Alpha}{\mathord}{ucgreekletters}{`A}
\DeclareMathSymbol{\Beta}{\mathord}{ucgreekletters}{`B}
\DeclareMathSymbol{\Gamma}{\mathord}{ucgreekletters}{`G}
\DeclareMathSymbol{\Delta}{\mathord}{ucgreekletters}{`D}
\DeclareMathSymbol{\Epsilon}{\mathord}{ucgreekletters}{`E}
\DeclareMathSymbol{\Zeta}{\mathord}{ucgreekletters}{`Z}
\DeclareMathSymbol{\Eta}{\mathord}{ucgreekletters}{`H}
\DeclareMathSymbol{\Theta}{\mathord}{ucgreekletters}{`J}
\DeclareMathSymbol{\Iota}{\mathord}{ucgreekletters}{`I}
\DeclareMathSymbol{\Kappa}{\mathord}{ucgreekletters}{`K}
\DeclareMathSymbol{\Lambda}{\mathord}{ucgreekletters}{`L}
\DeclareMathSymbol{\Mu}{\mathord}{ucgreekletters}{`M}
\DeclareMathSymbol{\Nu}{\mathord}{ucgreekletters}{`N}
\DeclareMathSymbol{\Xi}{\mathord}{ucgreekletters}{`X}
\DeclareMathSymbol{\Omicron}{\mathord}{ucgreekletters}{`O}
\DeclareMathSymbol{\Pi}{\mathord}{ucgreekletters}{`P}
\DeclareMathSymbol{\Rho}{\mathord}{ucgreekletters}{`R}
\DeclareMathSymbol{\Sigma}{\mathord}{ucgreekletters}{`S}
\DeclareMathSymbol{\Tau}{\mathord}{ucgreekletters}{`T}
\DeclareMathSymbol{\Upsilon}{\mathord}{ucgreekletters}{`U}
\DeclareMathSymbol{\Phi}{\mathord}{ucgreekletters}{`F}
\DeclareMathSymbol{\Chi}{\mathord}{ucgreekletters}{`Q}
\DeclareMathSymbol{\Psi}{\mathord}{ucgreekletters}{`Y}
\DeclareMathSymbol{\Omega}{\mathord}{ucgreekletters}{`W}
\makeatletter
\newtheoremstyle{frenchplain}
  {12pt plus 4pt minus 2pt}% space before
  {6pt}% space after
  {}% body font
  {}% indent
  {\rmfamily\itshape}% header font
  {.}% punctuation
  {.5em}% after theorem header
  {}% header specification (empty for default)
\makeatother

\theoremstyle{frenchplain}
\newtheorem{thm}{Theorem}[section]
\newtheorem{prop}[thm]{Proposition}
\newtheorem{lem}[thm]{Lemma}
\newtheorem{cor}[thm]{Corollary}
\newtheorem*{thm*}{Theorem}
\newtheorem*{conj*}{Conjecture}

\usepackage{graphicx}
\usepackage{tikz}
\usepackage{tikz-cd}
\usetikzlibrary{matrix,arrows}
\usetikzlibrary{shapes.geometric}
\usepgflibrary[arrows]
\usepgflibrary{arrows}
\usetikzlibrary[arrows]
\usetikzlibrary{decorations.pathmorphing}
\usetikzlibrary{decorations.markings}       
\tikzset{
  graphnode/.style = {align=center, inner sep=0pt, scale=0.3, text centered,
    font=\sffamily},
  vi/.style = {graphnode, circle, white, font=\sffamily\bfseries, draw=black,
    fill=black, text width=1em},
  ve/.style = {graphnode, circle, draw=black, 
    text width=1em, thick},
  vx/.style = {graphnode, draw=white,
    minimum width=0em, minimum height=0em},
  vb/.style = {graphnode, diamond, white, draw=black, 
    minimum width=1em, minimum height=1em, thick},
}
\tikzset{every loop/.style={}}  
\tikzset{commutative diagrams/arrow style=math font}
\newcommand{\pent}{\hspace{.4pt}
  \begin{tikzpicture}[font=\scriptsize, baseline=-.7ex] 
    \draw [use as bounding box, transparent, white] (-2mm,-2mm) rectangle (2mm,2mm);
    \node[regular polygon, shape border rotate=36, inner sep=.97mm,
          regular polygon sides=5, minimum size=1mm, draw] at (0,0) (A) {};
    \path (A.corner 1) edge node[auto] {} (A.corner 4)
          (A.corner 2) edge node[auto] {} (A.corner 5);
  \end{tikzpicture}\hspace{-.7pt}
}
\newcommand{\triang}{\hspace{.5pt}
  \begin{tikzpicture}[font=\scriptsize, baseline={([yshift=-.65ex]current bounding box.center)}]
    \node[regular polygon, shape border rotate=60, inner sep=.56mm,
          regular polygon sides=3, minimum size=1mm, draw] at (0,0) (A) {};
  \end{tikzpicture}\hspace{.1pt}
}

\newcommand*{\Cdot}[1][1.25]{%
  \mathpalette{\CdotAux{#1}}\cdot%
}
\newdimen\CdotAxis
\newcommand*{\CdotAux}[3]{%
  {%
    \settoheight\CdotAxis{$#2\vcenter{}$}%
    \sbox0{%
      \raisebox\CdotAxis{%
        \scalebox{#1}{%
          \raisebox{-\CdotAxis}{%
            $\mathsurround=0pt #2#3$%
          }%
        }%
      }%
    }%
    \dp0=0pt %
    \sbox2{$#2\bullet$}%
    \ifdim\ht2<\ht0 %
      \ht0=\ht2 %
    \fi
    \sbox2{$\mathsurround=0pt #2#3$}%
    \hbox to \wd2{\hss\usebox{0}\hss}%
  }%
}

\usepackage{etoolbox}
\appto{\proof}{\unskip} %OBS! Counteracts effect of parskip on \proof
\usepackage{mathtools}

\DeclareMathOperator{\dr}{\mathit{d}}
\DeclareMathOperator{\Reg}{\mathsf{Reg{\mkern 1mu}}}
\DeclareMathOperator{\Regr}{\mathsf{Regr{\mkern 1mu}}}
\DeclareMathOperator{\Spec}{\mathsf{Spec{\mkern 1mu}}}
\newcommand{\gr}[1]{\mathsf{gr}_{#1}{\mkern 2mu}}
\newcommand{\bbrackets}[1]{[\hspace{-1.3pt}[#1]\hspace{-1.3pt}]}
\newcommand{\Bbrackets}[1]{\bigl[\hspace{-2pt}\bigl[#1\bigr]\hspace{-2pt}\bigr]}
\newcommand{\susp}[1]{{\textstyle{s^{\hspace{.1pt}#1}}}\hspace{.3pt}}
\DeclareMathOperator{\Res}{\mathsf{Res}}
\DeclareMathOperator{\reg}{\mathsf{reg}}
\DeclareMathOperator{\id}{\mathsf{id}}
\newcommand{\Bimod}[1]{\operatorname{\textbf{inf\;}{\mathit{#1}}-\textbf{biMod}}}
\newcommand{\Mod}[1]{\operatorname{\textbf{Mod}-\mathit{#1}}}
\newcommand{\Coll}{\textbf{Coll}}
\newcommand{\dgV}{\textbf{dg\,Vect}}
\newcommand{\Md}{M^{\delta}}
\newcommand{\antishriek}{\hspace{.3pt}\text{!`}}

\newcommand{\Map}{\mathsf{Map}}
\newcommand{\Def}{\mathsf{Def}}
\newcommand{\Der}{\mathsf{Der}}
\newcommand{\As}{\mathsf{As}}
\newcommand{\Ass}{\mathsf{Ass}}
\newcommand{\Com}{\mathsf{Com}}
\newcommand{\Lie}{\mathsf{Lie}}
\newcommand{\BVGra}{
  \mathsf{Gra}^{\hspace{-5pt}\scalebox{.5}{
    \begin{tikzpicture}[baseline=0ex,shorten >=0pt,auto,node distance=.4cm]
      \node[ve] (w) {};
      \path[thick]
        (w) edge[out=135, in=45, loop] (w); 
    \end{tikzpicture}
  }\hspace{-4.5pt}}
}
\newcommand{\BVGraphs}{\mathsf{BVGraphs}}
\newcommand{\Graphs}{\mathsf{Graphs}}

\newcommand{\Ger}{\mathsf{Ger}}

\newcommand{\bS}{{\mkern 1mu}\mathbb{S}}
\newcommand{\Conv}{\mathbin{\scalebox{1.5}{\raisebox{-0.2ex}{$\ast$}}}}
\newcommand{\sconv}{{\mkern 1mu}\scalebox{1}{\raisebox{-0.1ex}{$\ast$}}{\mkern 1mu}}
\makeatletter
\newcommand{\superimpose}[2]{%
  {\ooalign{$#1\@firstoftwo#2$\cr\hfil$#1\@secondoftwo#2$\hfil\cr}}
}
\makeatother
\newcommand{\infcirc}{\mathbin{\mathpalette\superimpose{{\circ}{\cdot}}}}

\newcommand{\rO}{\mathscr{O}}
\newcommand{\rM}{\mathscr{M}}
\newcommand{\rW}{\mathscr{W}}

\newcommand{\C}{\mathbb{C}}
\newcommand{\N}{\mathbb{N}}
\newcommand{\Z}{\mathbb{Z}}
\newcommand{\Q}{\mathbb{Q}}
\newcommand{\R}{\mathbb{R}}

\title{The Grothendieck-Teichmueller Lie algebra\\ and Brown's dihedral moduli spaces}
\author{Johan Alm\\
        Stockholm University\\
        E-mail: \texttt{alm@math.su.se}
        }
\date{\today}
\makeatletter
\def\blfootnote{\gdef\@thefnmark{}\@footnotetext}
\makeatother

\begin{document}
\frenchspacing%\RaggedRight
\maketitle\thispagestyle{fancy}

%\section*{Abstract}
\begin{abstract}
We prove that the degree zero Hochschild-type cohomology of the homology operad of Francis Brown's dihedral moduli spaces is equal to the Grothendieck-Teichmueller Lie algebra plus two classes. This result significantly elucidates the (in part still conjectural) relation between the Grothendieck-Teichmueller Lie algebra and (motivic) multiple zeta values.
\end{abstract}

%\blfootnote{This work was supported by ERC-StG-759082.}

\section*{Introduction}
F.~Brown defined an affine partial compactification \({\displaystyle M^{\delta}_{0,n+1}}\) of the moduli space \(M_{0,n+1}\) of smooth genus zero curves with \(n+1\) marked points, in his pioneering work \cite{Brown09} on multiple zeta values. The collection \(M^{\delta}\) of these varieties naturally assemble as a nonsymmetric operad, such that the corresponding homology operad \(H_*(M^{\delta})\) (with rational coefficients) has a morphism \(\As\to H_*(M^{\delta})\) from the associative operad, mapping the binary generator to the class of the point \(\{\mathsf{pt}\} = \Md_{0,2+1}\). We prove that the degree zero cohomology of the associated Hochschild-type cohomology complex equals the Grothendieck-Teichm\"uller Lie algebra plus two classes:
\[
  R^0\Der\bigl(\As,H_*(M^{\delta})\bigr) \cong \Q\triang\oplus\Q\pent\oplus\mathfrak{grt}_1.
\]
With the conventions adopted in the paper, the result can also be expressed by saying that the degree one cohomology of the deformation complex of the induced morphism \(\As_{\infty}\to H_*(M^{\delta})\) equals \(\Q\pent\oplus\mathfrak{grt}_1\). The graphical notation in the formula stems from certain diagrammatic techniques used in the proof. Topologically, \(\triang\) is given by the fundamental class of \(\smash{\Md_{0,2+1}}\) while \(\pent\,\) is a basis vector of \(H_2(\Md_{0,4+1})\).

Our result can be phrased without reference to operads, as follows. The moduli spaces of Brown form a cosimplicial variety, with coface maps by gluing the point \(\smash{M^{\delta}_{0,2+1}}\) (``doubling'' a marked point) and codegeneracy maps given by forgetting a marked point. The cohomology \(R^*\Der(\As,H_*(M^{\delta}))\) is the cohomology of the total complex
\[
  \prod_{n\geq 2} H_{*+n-2}(\Md_{0,n+1})\;\;\;\text{with cosimplicial differential}.
\]
However, the operadic perspective has an absolutely fundamental role in our arguments. Indeed, our paper is part of a growing literature on the surprising ubiquity of the Grothendieck-Teichm\"uller Lie algebra in various operadic deformation problems: we mention here, e.g., the early preprint \cite{Tamarkin02} by Tamarkin, Willwacher's seminal \cite{Willwacher15}, and the papers \cite{MerkulovWillwacher15, TurchinWillwacher18} by Willwacher and collaborators.

The presentation of the Grothendieck-Teichm\"uller Lie algebra implied by our result is in some respects more accessible than the usual. In particular, it seems easier to write down explicit elements of \(R^0\Der(\As,H_*(M^{\delta}))\) than of \(\mathfrak{grt}_1\), as it is usually defined. The result also elucidates the conjectural relationship between the Grothendieck-Teichm\"uller Lie algebra and the indecomposable quotient of the algebra of multiple zeta values, in that it by Brown's \cite{Brown09} almost tautologically implies an injection
\[
  \Q\pent\oplus\mathfrak{grt}_1 \hookleftarrow 
  \bigl(Z^{\mathfrak{m}}_+ / (Z^{\mathfrak{m}}_+)^2\bigr)^\vee
\]
from the dual of the indecomposable quotient of the algebra of motivic multiple zeta values, \(Z^{\mathfrak{m}}\). A corollary to Brown's \cite{Brown12} is that there exists a (non-canonical) isomorphism
\[
  \bigl(Z^{\mathfrak{m}}_+ / (Z^{\mathfrak{m}}_+)^2\bigr)^\vee \cong \Q\sigma_2 \oplus \Lie\bbrackets{\,\sigma_{2k+1}\mid k\geq 1\,},
\]
from which we deduce an injection of a free Lie algebra \(\Lie\bbrackets{\,\sigma_{2k+1}\,}\) into \(\mathfrak{grt}_1\). Such an injection is already known to exist; our result just offers a slightly new perspective, a perspective that may be of interest for future work on the Drinfeld conjecture, which states that this injection is, in fact, also surjective.\footnote{\; The conjecture in the presently cited form is somewhat of an anachronism. Deligne and Ihara conjectured that the Galois group of mixed Tate motives unramified over the integers embeds in the automorphism group of the unipotent fundamental group(oid) of \(M_{0,4}\). Drinfeld defined a certain subgroup of that automorphism group, the Grothendieck-Teichm\"uller group, and conjectured that it was the image of the motivic Galois group. Simplifying somewhat, one might say that the Deligne-Ihara conjecture stated that the free Lie algebra on odd generators injects into the algebra of formal Lie series in two variables equipped with the Ihara bracket (which is now known), whereas it was Drinfeld who offered \(\mathfrak{grt}_1\) as a conjectural presentation of the image.} In \ref{sec:conjecture} we use our results to give a very concrete combinatorial and cohomological reformulation of this long-standing conjecture.

The structure of the paper is as follows. Section 1 explains notation and recalls some preliminaries on deformations of operad morphisms. Section 2 introduces Brown's moduli spaces and associated objects. Section 3 briefly recalls the definition of the Grothendieck-Teichm\"uller Lie algebra and concludes with the construction of a map from \(\mathfrak{grt}_1\) into the operadic Hochschild cohomology; the map which we are to prove is an isomorphiusm. A key role here is played by a novel nonsymmetric operad structure on \(H_*(M_0)\) (it is not the Gravity operad), introduced by the author in \cite{Alm16} but given further detail and a conceptual clarification in terms of logarithmic geometry in this paper. Section 4 is a detour on the homology of the framed little disks (the Batalin-Vilkovisky operad), inspired by the work of Kontsevich \cite{Kontsevich99} and Willwacher \cite{Willwacher15}. Section 5 concludes, using the results of section 4, the proof of our main theorem. Section 6 is a short exposition on the implied relationship between the Grothendieck-Teichm\"uller Lie algebra and multiple zeta values. Finally, section 7 is more or less an extended appendix, devoted to clarifying an argument by Tamarkin \cite{Tamarkin02} that we use in section 3.

\vspace{3mm}
\noindent
\textit{Acknowledgement}. It is a pleasure to thank Dan Petersen for insightful comments, Francis Brown for pointing out the dubious usage of the term ``Deligne-Drinfeld conjecture'' in the first version of the preprint, and Anton Khoroshkin for finding the gap in \cite{Alm16} that forced me to improve the arguments of the present paper. 

\section{Notation and terminology}

Throughout, \([n]\) is the set \(\{1,\dots, n\}\). The cardinality of a finite set \(I\) is written \(\#I\) and we sometimes write \(I+J\) for disjoint union of finite sets.

We use cohomological convention for differentials, so all complexes that are traditionally computed with homological differentials have reversed grading. We consider the category of differential graded rational vector spaces as a symmetric monoidal category in the usual way with Koszul sign rules. The internal hom-functor is denoted \(\Map\) and suspensions by \(\susp{d}\), so \((\susp{d}V)^k = V^{d+k}\) for a differential graded vector space \(V\).

For a Malcev Lie algebra \(\mathfrak{g}\), we write \(C_*(\mathfrak{g})\) for the symmetric coalgebra \(\hat{S}(\susp{}\mathfrak{g})\) completed with respect to the tensor length and equipped with the Chevalley-Eilenberg differential, and we similarly write \(C^*(\mathfrak{g})\) for the dual continuous cochain complex, i.e., the symmetric algebra \(S(\susp{-1}\mathfrak{g}^\vee)\) on the desuspension of the continuous linear dual, with Chevalley-Eilenberg cochain differential. In all instances of the paper where these onventions apply it just so happens that \(\mathfrak{g}\) is isomorphic to the weight completion \(\prod_{i\geq 1}\gr{i}(\mathfrak{g})\) of its associated graded \(\bigoplus_{i\geq 1}\gr{i}(\mathfrak{g})\), and then the dual \(\mathfrak{g}^{\vee}\) is the sum \(\bigoplus_{i\geq 1}\gr{i}(\mathfrak{g})^{\vee}\) of the (ordinary) linear duals of each weight-graded piece.

For notation on operads we mostly adhere to the book \cite{LodayVallette12} by Loday and Vallette, a reference that contains details on all the foundational results on operads that we shall take for granted. A \emph{collection} of objects (in some category) is an indexed family \(\{U(n)\}_{n\geq 1}\) of objects (in that category). A collection is \emph{symmetric} if each \(U(n)\) has an action of the symmetric group \(\bS_n\) on \(n\) letters. We refer to the index \(n\) of \(U(n)\) as the \emph{arity}. Note that we do not allow collections to have something in arity zero. (Symmetric) collections (in a given category) constitute a category. Tensor products, cartesian products, direct sums, and etc, are always understood arity-wise, e.g., \((-)\otimes(-)\) is used for the arity-wise tensor product of collections. The \emph{plethysm} monoidal product on collections is written \((-)\circ(-)\) and the \emph{Day convolution} \((-)\Conv(-)\). The plethysm unit is denoted \(\mathbb{I}\). We denote (co)augmentation (co)kernels by a subscript \(+\)-sign (rather than an overline, as \cite{LodayVallette12} does), and write \((-)\infcirc(-)\) instead of \((-)\circ_{(1)}(-)\) for the \emph{infinitesimal plethysm}. The book \cite{LodayVallette12} refers to symmetric collections as \(\bS\)-modules, the arity-wise tensor product as the Hadamard product and the Day convolution as, simply, the tensor product. We assume all (co)operads \(E\) to be (co)augmented and have a splitting \(E=\mathbb{I}\oplus E_+\), so that in effect we can drop the distinction between (co)operads and pseudo-(co)operads. If there is little risk of confusion we sometimes refer to nonsymmetric (co)operads simply as ``(co)operads''.

In the following subsections we fix further details on our conventions.

\subsection{Differential graded operads}

By \emph{differential graded operads} we mean operads in the category of differential graded vector spaces. If \(P\) is a differential graded operad then, in appropriate cases, the collection of the dual spaces \(P(n)^{\vee}\) will be a differential graded cooperad (and vice versa, cooperads can under appropriate assumptions be dualised to operads). We denote this linearly dual cooperad \(\mathit{coP}\).

Given a differential graded collection \(P\) we define its \emph{operadic suspension} \(\Sigma P\) by \(\Sigma P(n) = s^{1-n}\mathsf{sgn}_n\otimes P(n)\). If \(P\) is an operad or cooperad, then \(\Sigma P\) is as well. The convention is such that \(\Sigma\mathsf{End}(V) = \mathsf{End}(sV)\). If \(P\) is a nonsymmetric collection, we simply omit the sign representation and only shift degrees in the definiton of \(\Sigma P\).

\subsection{Cohomology and deformation theory}

An \emph{infinitesimal bimodule} for an operad \(P\) is a collection \(M\) equipped with left and right actions
\[
  \lambda : P \infcirc M \to M, \;\;\; \rho: M\circ P \to M,
\]
satisfying a list of identities that amount to saying \(P\oplus M\) is an operad, with the property that the composite of at least two elements coming from \(M\) vanishes.

\subsubsection{Bimodule cohomology}

Let \(P\) be a differential graded operad and \(M\) be an infinitesimal \(P\)-bimodule. The \emph{cohomology of\, \(P\) with coefficients in \(M\)} is the total right derived functor of \(M\mapsto \Der(P,M)\). When \(P\) is a Koszul operad, the Koszul resolution \(P_{\infty}\) leads to a tractable complex: \(\Der_{P}(P_{\infty}, M)\cong \Map_{\bS}(\susp{-1}P^{\antishriek}_+, M)\) with an extra differential \(\delta\) defined by the module structure. We denote this complex
\[
  C(P,M) = \Map_{\bS}(s^{-1}P^{\antishriek}_+, M)\;\;\text{with}\;\delta.
\]
We do not give a general formula for the differential as in this paper we will only compute with \(P=\Ass\) or \(P=\Com\) and then the formula for the differential has an especially simple form. 

\subsubsection{Deformations of morphisms}

An operad morphism \(g:O \to P\) makes \(P\) an infinitesimal \(O\)-bimodule. The \emph{deformation complex} of the map \(g\), written \(\Def(g:O\ \to P)\), is under our conventions closely related to the cohomology of \(O\) with coefficients in \(P\) but not equal to it. In case \(O\) is Koszul:
\[
  \Def(O \xrightarrow{g} P) = \Map_{\bS}(O^{\antishriek}, P)\;\;\text{with}\; d_g.
\]
The differential is the same up to sign, but the degree is shifted and we do not restrict to the augmentation ideal \(O^{\antishriek}_+\). The reason for the degree-shift is that \(\Map_{\bS}(O^{\antishriek}, P)\) is then naturally a differential graded Lie algebra (with bracket of degree zero), such that Maurer-Cartan elements are morphisms \(O_\infty\to P\). The extra differential is given in terms of the Lie bracket by \(\dr_g = [g,-]\). The reason for not restricting to the augmentation ideal is purely a matter of technical convenience.

\subsubsection{Preserves quasi-isomorphisms}\label{sec:defquism}

Assume given a quasi-isomorphism of \(P\)-bimodules,  \(M\to N\), for a Koszul operad \(P\). We can form complete (descending) filtrations on \(C(P,M)\) and \(C(P,N)\) by arity. Then, the induced morphism
\[
  E_0\, C(P,M) \to E_0\, C(P,N)
\]
is a quasi-isomorphism. But this means that the associated graded for the arity-filtration on the mapping cone of \(C(P, M)\to C(P,N)\) is acyclic; hence the mapping cone is acyclic. We conclude that the functor \(C(P,-)\) respects quasi-isomorphisms. The same is true for deformation complexes.

\subsubsection{Hochschild cohomology}

Because the symmetric associative operad is the free symmetric extension of the nonsymmetric associative operad,
\begin{align*}
  C(&\Ass, M) = \prod_{n\geq 2}\, \susp{2-n} \mathsf{sgn}_n\,\bS_n\otimes_{\bS_n}\! M(n) \\ 
  &\cong \prod_{n\geq 2}\,\susp{2-n} M(n) = C(\As, M).
\end{align*}
The differential is a cosimplicial differential (on factor of arity \(n\)) 
\[
  \delta = -\sum_{i=0}^{n+1} (-1)^i\delta_i,
\]
with \(\delta_0 = m_2\circ_2(-)\) and \(\delta_{n+1} = m_2\circ_1(-)\) the two left actions of the binary generator of \(\Ass\) and \(\delta_i = (-)\circ_i m_2\), for \(i=1,\dots,n\), the right actions. We shall refer to the cochain complexes \(C(\Ass, M)\) as \emph{Hochschild cohomology complexes} and to the differential as the \emph{Hochschild differential}.

If \(M = \mathsf{End}(A)\) for an associative algebra \(A\), then the Hochschild cochain complex of \(M\) recovers the (truncation of the) classical Hochschild cochain complex of \(A\).

\subsubsection{Harrison cohomology}

By the Koszul duality for the commutative operad, we have that
\[
  C(\Com, M) = \Map_{\bS}(\susp{-1}\Com^{\antishriek}_+, M)
             = \prod_{n\geq 2} \susp{2-n}\mathsf{sgn}_n\,\Lie(n)\otimes_{\bS_n}\! M(n),
\]
with a differential defined by the module structure on \(M\). The arity components of the Lie operad embed into the components of the associative operad, and the image
\[
  \Map_{\bS}(\susp{-1}\Com^{\antishriek}_+, M) \subset \Map_{\bS}(s^{-1}\Ass^{\antishriek}_+, M)\]
can be characterised as the space of maps that vanish on all signed nontrivial shuffles. Under this identification \(C(\Com, M)\) is a subcomplex of the Hochschild cochain complex \(C(\Ass, M)\), considering \(M\) as an infinitesimal bimodule for the associative operad via the canonical \(\Ass\to\Com\). We call the complex \(C(\Com, M)\) the \emph{Harrison cochain complex}.

\subsubsection{Hodge-type decomposition}\label{sec:HarrisonHodge}

In more generality, the group algebra \(\Q\bS_n\) has a \(n\) orthogonal idempotents
\[
  1 = e_n^1 + \dots + e_n^k
\] 
sometimes called the \emph{Eulerian idempotents}. These define a decomposition of any \(\bS_n\)-module. In particular, one has for any infinitesimal bimodule \(M\) of the associative operad, a decomposition
\[
  C(\Ass, M) = \prod_{n\geq 2} \bigoplus_{k=1}^n e_n^k\,\susp{2-n} M(n),
\]
as a graded vector space, or,
\[
  C(\Ass, M) = \bigoplus_k e^k\,C(\Ass, M).
\]
If the module structure happens to be defined by the canonical morphism \(\Ass\to\Com\) and an infinitesimal \(\Com\)-bimodule-structure on \(M\), then the above sum is actually a decomposition as complexes, i.e., in this case all the Eulerian idempotents commute with the Hochschild differential. This decomposition is called the \emph{Hodge decomposition of Hochschild cohomology}. Moreover,
\[
  e^1\,C(\Ass, M) = C(\Com, M)
\]
is the Harrison cochain complex. In slightly more detail, \(e^1_n\) can be seen as projection onto the subspace \(\mathsf{coLie}(n)\subset\mathsf{coAss}(n)\).

Explicit formulas for the Eulerian idempotents are complicated. In this paper only the following three are explicitly used:
\begin{align*}
  e^1_3 &= \frac{1}{6}\hspace{.4pt}\biggl(2\id + (12) + (23) - (123) - (132) - 2(13)\biggr), \\
  e^2_3 &= \frac{1}{2}\hspace{.4pt}\biggl(\id + (13)\biggr), \\
  e^3_3 &= \frac{1}{6}\hspace{.4pt}\biggl(\id - (12) - (13) - (23) + (123) + (132)\biggl).
\end{align*}
More generally, \(e^{\hspace{.5pt}n}_n\) is the projection to the totally antisymmetric part.

\section{Preliminaries on moduli spaces}

Write \(\chi_1(n)\) for the set of \emph{chords} on the regular \((n+1)\)-gon with its sides labelled counter clockwise \(1,\dots, n+1\). Every chord divides the set \([n+1]\) into a disjoint union of the two sets of indices falling on either side. Thus, by associating to a chord the pair \(\{i,j\}\) where \(i\) and \(j\) are the smallest and largest indices \emph{not} falling on the same side as \(n+1\), we can equivalently regard the set of chords \(\chi_1(n)\) as the set of pairs \(\{i,j\}\subset [n]\) such that \(i\neq j\) and \(\{i,j\}\neq \{1,n\}\).

There is an evident notion of when two chords \emph{cross}, at least if we agree that a chord does not cross itself. Say that two subsets \(A\) and \(B\) of \(\chi_1(n)\) are \emph{completely crossing} if for every \(a\in A\), all \(b\in B\) cross \(a\), and vice versa, for every \(b\in B\), all chords in \(A\) cross \(b\).

\subsubsection{The open moduli space}

For \(n\geq 2\), define the \emph{open moduli space} of genus zero curves with \(n+1\) marked points to be the affine variety
\[
  M_{0,\hspace{.4pt}n+1} = \Spec\Q[\,{\textstyle u^{\pm 1}_c} \mid c\in\chi_1(n)\,] / \langle\, R\,\rangle,
\]
where \(R\) is the set of relations that
\[
  \prod_{a\in A} u_a + \prod_{b\in B} u_b = 1
\]
for all completely crossing subsets \(A, B\subset \chi_1(n)\). For example,
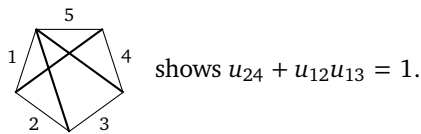
\[
\begin{tikzpicture}[baseline={([yshift=-.5ex]current bounding box.center)}]
\node[regular polygon, shape border rotate=36, regular polygon sides=5, minimum size=1.5cm, draw] at (5*2,0) (A) {};
\foreach \i in {1,...,5} {
    \coordinate (ci) at (A.corner \i);
    \coordinate (si) at (A.side \i);
    \node at (ci) {};
    %\path (A.center) -- (ci) node[pos=1.2] {};
    \path (A.center) -- (si) node[pos=1.3] {\scriptsize{$\i$}}; 
 }
 \path[thick] (A.corner 1) edge (A.corner 3)
              (A.corner 1) edge (A.corner 4)
              (A.corner 2) edge (A.corner 5);
\end{tikzpicture}\;\;\text{shows}\ u_{24} + u_{12}u_{13} = 1.
\]
This presentation is due to \cite{Brown09} and it is equivalent to the more classical definition
\[
  M_{0,\hspace{.4pt}n+1} = \bigl((P^1)^{\hspace{.5pt}n+1})\setminus \textsf{diagonals}\bigr) / \mathsf{PGL}_2,
\]
via the identification that a generator \(u_{ij}\) of the coordinate ring denotes the cross-ratio function\footnote{\;We caution the reader that our numbering of the chords does \emph{not} agree with the one used in \cite{Brown09}: our \(u_{ij}\) is Brown's \(u_{i-1\,j}\).}
\[
  u_{ij} = [\,i-1\;i\mid j+1\; j\,] =\frac{(z_i - z_j)(z_{i-1} - z_{j+1})}{(z_{i-1} - z_j)(z_i - z_{j+1})}.
\]
The space \(M_{0,\hspace{.4pt}n+1}\) has an action of \(\bS_{n+1}\) by permuting marked points. In particular, it has an action by the subgroup \(\bS_n\) that fixes the point \(z_{n+1}\). Note that the action on the coordinate ring is not simply \(u_{ij}\cdot\sigma = u_{\sigma(i)\sigma(j)}\).

\subsubsection{The dihedral moduli space}

Brown defined a partial compactification of the open moduli space by
\[
  \Md_{0,n+1} = \Spec\Q[\,u_c \mid c\in\chi_1(n)\,] / \langle\, R\,\rangle,
\]
where \(R\) is the same set of relations as that defining the open moduli space. We shall refer to \(\Md_{0,n+1}\) as \emph{Brown's dihedral moduli space}.

Brown's dihedral moduli space has an action of the dihedral group \(D_{2n+2}\), but not of the symmetric group \(\bS_{n+1}\).

\subsubsection{Easier notation}

Set \(M_0(n) = M_{0,n+1}\) and \(M^{\delta}(n) = \Md_{0,n+1}\).

\subsubsection{Operad structure}

Brown's dihedral moduli spaces naturally assemble as a nonsymmetric (pseudo-)operad in the category of affine varieties. The definition of a partial composition
\[
  \circ_i: M^{\delta}(n/I)\times M^{\delta}(I) \to M^{\delta}(n)
\]
is as follows. Here \(n/I\) is the set with \(n-\#I + 1\) elements, where all elements in the subset \(I\subset[n]\) are identified. Since the operad is nonsymmetric, \(I\subset [n]\) is a connected subinterval; so it corresponds to a chord \(c\) on the \((n+1)\)-gon (explicitly, \(c=\{i,j\}\) where \(i\) and \(j\) are the endpoints of \(I\)). We define a morphism of algebras
\[
  \Delta_c : \rO(M^{\delta}(n)) \to \rO(M^{\delta}(n/I))\otimes \rO(M^{\delta}(I)), 
\]
\[
  \Delta_c(u_a) = \left\{
    \begin{aligned}
     &\, 0   && \text{if}\ a=c, \\
     &\, 1\otimes 1   && \text{if}\ a\ \text{crosses}\ c, \\
     &\, 1\otimes u_a && \text{if}\ a\subset I, \\
     &\, u_a\otimes 1 && \text{otherwise}.
  \end{aligned}\right.
\]
Here we abuse notation on the last line and write \(a\) also for the image of \(a\) in \(n/I\). Informally, \(\Delta_c\) acts as if cutting the \((n+1)\)-gon into two smaller polygons, setting to \(1\) all chords that it cuts and reinterpreting all other chords as chords on whichever subpolygon that contains them. Cutting polygons is equivalent to splitting planar trees. The operations \(\Delta_c\) thus almost tautologically satisfy the axioms for a nonsymmetric cooperad in the category of commutative algebras.

In fact, as \emph{sets} of complex points
\[
  M^{\delta}(\C)(n)\, = \bigsqcup_{I\subset [n]} M_0(\C)(n/I)\times M_0(\C)(I),
\]
or even more succinctly, \(M^{\delta}(\C)\) is the free nonsymmetric operad of sets generated by the collection of sets \(M_0(\C)\). Thus, the partial compositions of the dihedral moduli spaces are inclusions \(M^{\delta}(n/I)\times M^{\delta}(I) \subset M^{\delta}(n)\) of closed boundary strata, and those strata correspond bijectively to chords \(c\in\chi_1(n)\).

\subsection{Logarithmic structure}\label{sec:logstruct}

The nonsymmetric operad structure on the dihedral moduli spaces can be extended to the category of logarithmic varieties. This result is new. We intend to develop this further in future work and content ourselves here with giving only the core definitions and lemmata. The reason we include this section is that it clarifies a few constructions that may otherwise seem suspiciously ad hoc. Define the monoid
\[
  W(n) = \N \oplus \N^{\,\chi_1(n)}
\]
and the monoid morphism
\[
  u: W(n) \to \rO(M^{\delta}(n)),\ u\bigl(k, \kappa\bigr) = 0^k\cdot \prod_{c\in\chi_{\delta}} u_c^{\kappa(c)}.
\]
Then \(M^{\delta}(n)\) together with logarithmic structure \(\rW(n)\) associated to \(W(n)\) is a logarithmic scheme. It is defined over the logarithmic point \(\Spec(\Q,\N)\) in an obvious way: \(\N \to W(n), k \mapsto (k, 0)\) and the coordinate ring is an algebra over \(\Q\). The logarithmic scheme \((M^{\delta}(n),\rW(n))\) is the base change of the canonical logarithmic structure defined by the divisor \(M^{\delta}(n)\setminus M_0(n)\), to a scheme over \(\Spec(\Q,\N)\).

\subsubsection{Boundary inclusions}

We here show how to upgrade the boundary inclusions
\[
  \circ_i: M^{\delta}(n/I)\times_{\Spec\hspace{-.6pt}\Q} M^{\delta}(I) \to M^{\delta}(n)
\]
to morphisms of logarithmic varieties. The monoids \(W(n) = \N\oplus \N^{\,\chi_1(n)}\) are free monoids under the monoid \(\N\) that defines the logarithmic point. It follows that
\[
  W(n/I)\oplus_{\N} W(I) = \N\oplus \N^{\,\chi_1(n/I)} \oplus \N^{\,\chi_1(I)}.
\]
Let \(c\) be the chord on the \(n\)-gon corresponding to the composition \(\circ_i\) and note that we have a canonical inclusion \(\chi_1(n/I)\sqcup\chi_1(I)\hookrightarrow\chi_1(n)\). Define
\begin{align*}
  (\circ_i)^{\flat}: W(n) &\to W(n/I)\oplus_{\N} W(I),\\
  (k,\kappa) &\mapsto \bigl(k+\kappa(c),\kappa\vert_{\,\chi_1(n/I)}, \kappa\vert_{\,\chi_1(I)}\bigr).
\end{align*}

\begin{lem}
The monoid morphism \((\circ_i)^{\flat}\) extends the boundary inclusion to a morphism
\[
  \circ_i: (M^{\delta},\rW)(n/I)\times_{\Spec(\Q,\N)} (M^{\delta},\rW)(I) \to (M^{\delta},\rW)(n)
\]
of logarithmic varieties.
\end{lem}

\begin{proof}
We may work in global charts. On the one hand,
\[
  (\circ_i)^{\sharp}\circ u\bigl(k,\kappa\bigr) = 
  0^k\left\{
    \begin{aligned}
    &\;\; 0 &&\text{if}\ \kappa(c)\geq 1,\\
    &\, \prod_{a\in\chi_1(n/I)} u_a^{\kappa(a)}\otimes \prod_{b\in\chi_1(I)} u_b^{\kappa(b)}
      &&\text{otherwise}.
    \end{aligned}\right.
\]
On the other hand,
\[
 (u\oplus_{\N} u)\circ(\circ_i)^{\flat}\bigl(k, \kappa\bigr) =0^{k+\kappa(c)}\prod_{a\in\chi_1(n/I)} u_a^{\kappa(a)}\otimes \prod_{b\in\chi_1(I)} u_b^{\kappa(b)}.
\]
That \((\circ_i)^{\flat}\) is a map under \(\N\) is clear.
\end{proof}

\subsection{Cohomology algebras}

\subsubsection{Cohomology of the open moduli space}

Let \(\alpha_c = \dr\log u_c\) as a class in the algebraic de Rham complex of \(M_0(n)\). The rational, algebraic de Rham cohomology \(H^*(M_0(n))\) is the free graded commutative algebra generated by \(\{\,\alpha_c\mid c\in\chi_1(n)\,\}\) modulo the relations that
\[
  \biggl(\sum_{a\in A} \alpha_c\biggr)\wedge\biggl(\sum_{b\in B}\alpha_b\biggr) = 0
\]
for all completely crossing subsets \(A, B\subset \chi_1(n)\). Note that the forms \(\alpha_c\) satisfy these relations as forms and not just as cohomology classes. This presentation is due to \cite{Brown09} by Brown.

\subsubsection{A characterisation with residues}

There seems to be no simple presentation by generators and relations of the cohomology algebra of the dihedral moduli spaces. However, geometry lets us characterise the image of the canonical restriction \(H^*(M^{\delta}(n)) \to H^*(M_0(n))\) as the kernel
\[
  \bigcap_{c\in\chi_1(n)} \mathsf{Ker}(\Res_c)
\]
of all Poincar\'e residue maps to boundary strata. It follows from \cite{AlmPetersen17, DupontVallette17} that the restriction to the cohomology of the open moduli space is injective, so that the cohomology of the dihedral moduli space can actually be characterised as this joint kernel. We make no explicit use of this characterisation of the cohomology ring but it implicitly serves to justify the terminology of \emph{regularisation}, introduced in \ref{sec:regularisation}.

\subsubsection{Chord diagrams}

A \emph{chord diagram} on the \((n+1)\)-gon is a monic monomial in the free graded commutative algebra \(\Q[\,\chi_1(n)\,]\), where we give each generator chord \(c\) degree plus one. We can draw such a diagram as a collection of superposed chords. The ``graded commutative'' means that the chords in the diagram must be considered as ordered up to an even permutation and that double chords are impossible (such a diagram is zero because \(c\wedge c = 0\)). If we need to refer to the explicit indices of chords we will write them as subscripts, e.g, write \(c_{13}\wedge c_{24}\) rather than \(\{1,3\}\wedge\{2,4\}\). Below is an example showing our conventions. 
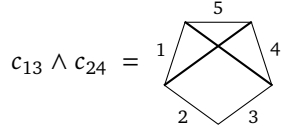
\[
c_{13}\wedge c_{24}\ =
\begin{tikzpicture}[baseline={([yshift=-.5ex]current bounding box.center)}]
\node[regular polygon, shape border rotate=36, regular polygon sides=5, minimum size=1.5cm, draw] at (5*2,0) (A) {};
\foreach \i in {1,...,5} {
    \coordinate (ci) at (A.corner \i);
    \coordinate (si) at (A.side \i);
    \node at (ci) {};
    %\path (A.center) -- (ci) node[pos=1.2] {};
    \path (A.center) -- (si) node[pos=1.3] {\scriptsize{$\i$}}; 
 }
 \path[thick] (A.corner 1) edge node[auto] {} (A.corner 4)
              (A.corner 2) edge node[auto] {} (A.corner 5);
\end{tikzpicture} 
\]
We always label the top side \(n+1\) and label counter-clockwise in the order of indices. Because of this the labelling of sides is superflous and can be suppressed.

Note that we have a surjection \(\Q[\,\chi_1\,]\to H^*(M_0)\), mapping \(c\) to \(\alpha_c\).

\subsubsection{Gravity chord diagrams}

Define a chord diagram to be a \emph{gravity chord diagram} if the following condition holds. For every pair of crossing chords in the diagram, consider the corresponding inscribed quadrilateral. The side of this quadrilateral that is opposite from the distinguished side \((n+1)\) of the polygon is not allowed to be a side of the polygon, nor is it allowed to be a chord in the diagram. 

The two forms of inadmissible chord diagrams are illustrated in figures \ref{grav1} and \ref{grav2}; the ``inscribed quadrilaterals'' mentioned in the definition are depicted by dotted lines, and the distinguished side of the polygon is the bold top side. Note that the inadmissible subdiagrams correspond to certain monomial factors.

\begin{figure}[!h]
	\centering
	\begin{minipage}{.45\linewidth}
		\centering
		\begin{tikzpicture}[font=\scriptsize, baseline={([yshift=-.5ex]current bounding box.center)}]
		\node[regular polygon, shape border rotate=45,
		      regular polygon sides=8, minimum size=2cm, draw] at (5*2,0) (A) {};
		%\foreach \i in {1,...,8} {
			%\coordinate (ci) at (A.corner \i);
			%\coordinate (si) at (A.side \i);
			%\node at (ci) {};
			%\path (A.center) -- (ci) node[pos=1.2] {(\i)};
			%\path (A.center) -- (si) node[pos=1.25] {}; 
		%}
		\path[densely dotted]
		(A.corner 2) edge node[auto] {} (A.corner 4)
		(A.corner 2) edge node[auto] {} (A.corner 7)
		(A.corner 5) edge node[auto] {} (A.corner 7);
		\path[thick] 
		(A.corner 2) edge node[auto] {} (A.corner 5)
		(A.corner 4) edge node[auto] {} (A.corner 7);
		%(A.corner 3) edge node[auto] {} (A.corner 5);
		\path[very thick]
		(A.corner 1) edge node[auto] {} (A.corner 8);
		%(A.corner 6) edge node[auto] {} (A.corner 8)
		%(A.corner 2) edge node[auto] {} (A.corner 8);
		\end{tikzpicture}
		
		\captionof{figure}{A chord diagram of the form $c_{ij}c_{jk}$ with $i<j<k$.}
		\label{grav1}
	\end{minipage}\hspace{0.05\linewidth}\begin{minipage}{.45\linewidth}
		\centering
		\begin{tikzpicture}[font=\scriptsize, baseline={([yshift=-.5ex]current bounding box.center)}]
		\node[regular polygon, shape border rotate=45,
		      regular polygon sides=8, minimum size=2cm, draw] at (5*2,0) (A) {};
		%\foreach \i in {1,...,8} {
			%\coordinate (ci) at (A.corner \i);
			%\coordinate (si) at (A.side \i);
			%\node at (ci) {};
			%\path (A.center) -- (ci) node[pos=1.2] {(\i)};
			%\path (A.center) -- (si) node[pos=1.25] {}; 
		%}
		\path[densely dotted]
		(A.corner 7) edge node[auto] {} (A.corner 5)
		(A.corner 1) edge node[auto] {} (A.corner 3)
		(A.corner 1) edge node[auto] {} (A.corner 7);
		\path[thick] 
		(A.corner 1) edge node[auto] {} (A.corner 5)
		(A.corner 3) edge node[auto] {} (A.corner 7)
		(A.corner 3) edge node[auto] {} (A.corner 5);
		\path[very thick]
		(A.corner 1) edge node[auto] {} (A.corner 8);
		%(A.corner 6) edge node[auto] {} (A.corner 8)
		%(A.corner 2) edge node[auto] {} (A.corner 8);
		\end{tikzpicture} 
		
		\captionof{figure}{A chord diagram $c_{ik}c_{jk}c_{jl}$ with indices $i<j<k<l$.} 
		\label{grav2}
	\end{minipage}
\end{figure}

\subsubsection{Prime chord diagrams}

Say that a chord in a diagram is \emph{residual} if it is not crossed by any other chord in the diagram, and we define a chord diagram to be \emph{prime} if it does not contain any residual chords.

\begin{lem}\label{lem:chordbasis}
Gravity chord diagrams constitute a basis for the cohomology of the open moduli space. The (sub)set of prime gravity chord diagrams is a basis for the cohomology of the dihedral moduli space.
\end{lem}
\noindent
The above is one of the main results of \cite{AlmPetersen17} by the author and Petersen.

\subsection{Cooperad structures}

\subsubsection{Regularised restriction}

Assume given a chord \(c\in\chi_1(n)\) and let \(I\subset [n]\) be the connected interval enclosed by \(c\). We define the \emph{regularised restriction} onto the boundary strata corresponding to \(c\) by
\begin{gather*}
  \Reg_c : H^*(M_0(n)) \to H^*(M_0(n/I))\otimes H^*(M_0(I)), \\
  \Reg_c(\alpha_b) = \left\{
    \begin{aligned}
     &\, 0   && \text{if}\ b=c, \\
     &\, 0   && \text{if}\ b\ \text{crosses}\ c, \\
     &\, 1\otimes \alpha_b && \text{if}\, a\subset I, \\
     &\, \alpha_b\otimes 1 && \text{otherwise}.
    \end{aligned}\right.
\end{gather*}
Here \(b\) in \(\alpha_b\otimes 1\) denotes the image of \(b\) in \(\chi_1(n/I)\), abusing notation. The strata is given by the equation \(u_c=0\) and forms that are regular on this strata are simply restricted, whereas \(\dr u_c /u_c\) is ``regularised'' to zero.

Diagrammatically, \(\Reg_c\) acts on chord diagrams by cutting along \(c\), sending all diagrams that contain \(c\) or a chord that crosses it to zero. Observe however that the resulting diagrams need not be gravity diagrams, even though the input diagram is. Of course, one can always use the algebra relations to rewrite any non-gravity chord diagram as a sum of gravity chord diagrams. For example,
\[
  \Reg_{23}\,
\begin{tikzpicture}[font=\scriptsize, baseline={([yshift=-.5ex]current bounding box.center)}]
  \node[regular polygon, shape border rotate=60,
        regular polygon sides=6, minimum size=1cm, draw] at (5*2,0) (A) {};
  \path[thick] (A.corner 1) edge node[auto] {} (A.corner 4)
               (A.corner 2) edge node[auto] {} (A.corner 5);
\end{tikzpicture}\,
=\,
\begin{tikzpicture}[font=\scriptsize, baseline={([yshift=-.5ex]current bounding box.center)}]
  \node[regular polygon, shape border rotate=36,
        regular polygon sides=5, minimum size=1cm, draw] at (5*2,0) (A) {};
  \path[thick] (A.corner 1) edge node[auto] {} (A.corner 3)
               (A.corner 2) edge node[auto] {} (A.corner 4);
\end{tikzpicture}
  \,\otimes
\begin{tikzpicture}[font=\scriptsize, baseline={([yshift=-.5ex]current bounding box.center)}]
  \node[regular polygon, shape border rotate=60,
        regular polygon sides=3, minimum size=.7cm, draw] at (5*2,0) (A) {};
\end{tikzpicture} \,
  =\,
\begin{tikzpicture}[font=\scriptsize, baseline={([yshift=-.5ex]current bounding box.center)}]
  \node[regular polygon, shape border rotate=36,
        regular polygon sides=5, minimum size=1cm, draw] at (5*2,0) (A) {};
  \path[thick] (A.corner 1) edge node[auto] {} (A.corner 4)
               (A.corner 2) edge node[auto] {} (A.corner 5);
\end{tikzpicture}
  \,\otimes
\begin{tikzpicture}[font=\scriptsize, baseline={([yshift=-.5ex]current bounding box.center)}]
  \node[regular polygon, shape border rotate=60,
        regular polygon sides=3, minimum size=.7cm, draw] at (5*2,0) (A) {};
\end{tikzpicture} 
\]
is the diagrammatic counterpart of
\[
  \Reg_{23}(\alpha_{13}\wedge\alpha_{24}) = \alpha_{12}\wedge\alpha_{23}\otimes 1  
  %(&= -\alpha_{12}\wedge\alpha_{24}\otimes 1\ ) \\
  = \alpha_{13}\wedge\alpha_{24}\otimes 1.
\]
Regularised restriction gives the cohomology \(H^*(M_0)\) of the open moduli spaces the structure of a nonsymmetric cooperad of graded commutative algebras.

When restricted to the subspace \(H^*(M^{\delta})\subset H^*(M_0)\) of prime diagrams the regularised restriction is simply the restriction, i.e., the canonical pullback on cohomology to boundary strata, which since \(M^{\delta}\) is a nonsymmetric operad of varieties defines a nonsymmetric cooperad of graded commutative algebras.

Regularised restriction has a maybe more conceptual definition in terms of logarithmic geometry. To any morphism \((X,\rM_X)\to (S,\rM_S)\) of logarithmic schemes one may associate a relative logarithmic de Rham complex
\[
  \Omega^*_{X/S}\langle\log\rM_X/\rM_S\rangle.
\]
The complex
\[
  \Omega^*_{M^{\delta}(n)/\Q}\langle\log\rW(n)/\N\rangle
\]
coincides with the classically defined complex of forms with logarithmic singularities on the boundary divisor \(M^{\delta}(n)\setminus M_0(n)\). In particular, the space
\[
 \Omega^*(M^{\delta},\rW)(n) = \Gamma\bigl(M^{\delta}(n),\Omega^*_{M^{\delta}(n)/\Q}\langle\log\rW(n)/\N\rangle\bigr)
\]
of global sections is generated over the coordinate ring \(\rO(M^{\delta}(n))\) by the one-forms \(\dr\log u_c = \dr u_c/u_c\).

By functoriality of logarithmic forms, we get an induced
\[
  (\circ_i)^* : \Omega^*(M^{\delta},\rW)(n) \to \Omega^*(M^{\delta},\rW)(n/I)\otimes \Omega^*(M^{\delta},\rW)(I).
\]
\begin{lem}
The pullback \((\circ_i)^*\) is given by
\[
  (\circ_i)^*(\dr\log u_a) = \left\{
  \begin{aligned}
  &\, 0 &&\text{if}\ a = c, \\
  &\, 0 &&\text{if}\ a\ \text{crosses}\ c, \\
  &\, \dr\log u_a \otimes 1 &&\text{if}\ a\in\chi_1(n/I),\\
  &\, 1\otimes \dr\log u_a &&\text{if}\ a\in\chi_1(I).
  \end{aligned}\right.
\]
In other words, \((\circ_i)^* = \Reg_c\) is the regularised restriction.
\end{lem}

\begin{proof}
Clear from the definition.
\end{proof}

\subsubsection{Relation to the Gravity operad}

Getzler showed in \cite{Getzler95} that \(s^{-1}H^*(M_0)\) is a cooperad with Poincar\'e residue as cocomposition, and christened it the \emph{Gravity cooperad}, denoted \(\mathsf{coGrav}\). We emphasise that the cooperad structure we consider here is different. The two are related by the formula
\[
  \Res_c = \Reg_c \circ \frac{\partial\;\;}{\partial\alpha_c}.
\]
The graded derivation \(\partial/\partial\alpha_c\) is not well-defined on the cohomology algebra, only on \(\Q[\,\chi_1(n)]\), but the ambiguity lies in the kernel of \(\Reg_c\) so the composite is well-defined. Informally, if regularised restriction is cutting along \(c\), then taking residue is removing \(c\) and then cutting. Here is an example:
\[
\Res_{34}\;
\begin{tikzpicture}[font=\scriptsize, baseline={([yshift=-.6ex]current bounding box.center)}]
  \node[regular polygon, shape border rotate=360/14,
        regular polygon sides=7, minimum size=1cm, draw] at (7*2,0) (A) {};
  \path[thick] (A.corner 1) edge node[auto] {} (A.corner 5)
               (A.corner 2) edge node[auto] {} (A.corner 7)
               (A.corner 3) edge node[auto] {} (A.corner 5);
\end{tikzpicture}
\; = \; 
\begin{tikzpicture}[font=\scriptsize, baseline={([yshift=-.6ex]current bounding box.center)}]
  \node[regular polygon, shape border rotate=360/6,
        regular polygon sides=6, minimum size=1cm, draw] at (6*2,0) (A) {};
  \path[thick] (A.corner 1) edge node[auto] {} (A.corner 4)
               (A.corner 2) edge node[auto] {} (A.corner 6);
\end{tikzpicture}
 \,\otimes 
\begin{tikzpicture}[font=\scriptsize, baseline={([yshift=-.6ex]current bounding box.center)}]
  \node[regular polygon, shape border rotate=360/6,
        regular polygon sides=3, minimum size=.7cm, draw] at (3*2,0) (A) {};
\end{tikzpicture}.
\]
Unlike regularised restriction, taking the residue of a gravity chord diagram always produces a tensor of gravity chord diagrams. In lieu of this, lemma \ref{lem:chordbasis} can be strengthened to:

\begin{thm}\label{thm:almpetersen}
The Gravity operad is freely generated as a nonsymmetric operad by the collection of sets of prime gravity chord diagrams.
\end{thm}
\noindent
Details can be found in \cite[section 2]{AlmPetersen17}.

\subsubsection{Relation to the Lie operad}\label{sec:liecorr}

The Gravity operad contains a copy of the Lie operad (up to an operadic suspension). More precisely, the gravity chord diagrams that correspond to forms of top degree can be interpreted as a basis of the Lie operad, and the prime gravity chord diagrams of top degree give a generating collection for the Lie operad as a nonsymmetric operad. The correspondence is simple, one essentially just reinterprets a chord \(c_{ij}\) as a bracketing of inputs \(i\) and \(j\), and adds an overall bracket \([1,n]\) corresponding to the distinguished side \(n+1\). For example:
\[
\begin{tikzpicture}[font=\scriptsize, baseline={([yshift=-.5ex]current bounding box.center)}]
  \node[regular polygon, shape border rotate=60,
        regular polygon sides=6, minimum size=1cm, draw] at (5*2,0) (A) {};
  \path[thick] (A.corner 1) edge node[auto] {} (A.corner 4)
               (A.corner 2) edge node[auto] {} (A.corner 5)
               (A.corner 2) edge node[auto] {} (A.corner 6);
\end{tikzpicture}\,
\sim \bigl[[1,3],[[2,4],5]\bigr]\in\Lie(4).
\]
This gives a different characterisation of prime gravity chord diagrams. 
Define \(L(n)\) to be the set of iterated binary bracketings of the indicies \(1,\dots , n\), subject to the following conditions: (i) Each index appears exactly once, and (ii) the smallest index in a bracket stands to the left and the largest to the right. For example, \([1,[2,3]]\) and \([[1,2],3]\) both lie in \(L(3)\), but neither \([2,[1,3]]\) nor \([[1,3],2]\) does.

Say that a binary bracket \(b\) (of bracketings) in an element of \(L(n)\) is \emph{connected} if the set of indices appearing inside \(b\) is a subinterval of \([n]\). Define the set \(P(n)\) of \emph{prime bracketings}, to be the subset of \(L(n)\) consisting of all those \(P\) with the property that only the outermost bracket is connected.

Then \(L(n)\) is equal to the set of top degree gravity chord diagrams and \(P(n)\) is equal to the set of top degree prime gravity chord diagrams.

\subsection{Regularisation}\label{sec:regularisation}

Setting all gravity chord diagrams that are not prime to zero defines a morphism of collections
\[
  \reg: H^*(M_0) \to H^*(M^{\delta})
\]
that we call \emph{regularisation}. Observe that it is left inverse to the canonical restriction. Note also that it does not respect the algebra structures. It is, furthermore, not a morphism of cooperads, as the following example shows:
\[
\Reg_{45}
\begin{tikzpicture}[baseline={([yshift=-.5ex]current bounding box.center)}]
\node[regular polygon, shape border rotate=360/14, regular polygon sides=7, minimum size=1cm, draw] at (5*2,0) (A) {};
\foreach \i in {1,...,7} {
    \coordinate (ci) at (A.corner \i);
    \coordinate (si) at (A.side \i);
    \node at (ci) {};
    %\path (A.center) -- (ci) node[pos=1.2] {};
    %\path (A.center) -- (si) node[pos=1.3] {\scriptsize{\i}}; 
 }
 \path[thick] (A.corner 1) edge (A.corner 3)
              (A.corner 1) edge (A.corner 6)
              (A.corner 4) edge (A.corner 7);
\end{tikzpicture}
 = \Bigl(
\begin{tikzpicture}[baseline={([yshift=-.5ex]current bounding box.center)}]
\node[regular polygon, shape border rotate=360/6, regular polygon sides=6, minimum size=1cm, draw] at (5*2,0) (A) {};
\foreach \i in {1,...,6} {
    \coordinate (ci) at (A.corner \i);
    \coordinate (si) at (A.side \i);
    \node at (ci) {};
    %\path (A.center) -- (ci) node[pos=1.2] {};
    %\path (A.center) -- (si) node[pos=1.3] {\scriptsize{\i}}; 
 }
 \path[thick] (A.corner 1) edge (A.corner 3)
              (A.corner 1) edge (A.corner 5)
              (A.corner 3) edge (A.corner 6);
\end{tikzpicture}
+
\begin{tikzpicture}[baseline={([yshift=-.5ex]current bounding box.center)}]
\node[regular polygon, shape border rotate=360/6, regular polygon sides=6, minimum size=1cm, draw] at (5*2,0) (A) {};
\foreach \i in {1,...,6} {
    \coordinate (ci) at (A.corner \i);
    \coordinate (si) at (A.side \i);
    \node at (ci) {};
    %\path (A.center) -- (ci) node[pos=1.2] {};
    %\path (A.center) -- (si) node[pos=1.3] {\scriptsize{\i}}; 
 }
 \path[thick] (A.corner 1) edge (A.corner 3)
              (A.corner 1) edge (A.corner 5)
              (A.corner 2) edge (A.corner 6);
\end{tikzpicture} 
\Bigr)\otimes\!\!
\begin{tikzpicture}[baseline={([yshift=-.5ex]current bounding box.center)}]
\node[regular polygon, shape border rotate=360/6, regular polygon sides=3, minimum size=.7cm, draw] at (5*2,0) (A) {};
\foreach \i in {1,...,3} {
    \coordinate (ci) at (A.corner \i);
    \coordinate (si) at (A.side \i);
    \node at (ci) {};
    %\path (A.center) -- (ci) node[pos=1.2] {};
    %\path (A.center) -- (si) node[pos=1.5] {\scriptsize{\i}}; 
 }
\end{tikzpicture}.
\]
The displayed shows \(\reg(\Reg_{45}(\alpha_{15}\alpha_{12}\alpha_{46})) = \alpha_{14}\alpha_{13}\alpha_{25}\otimes 1\), corresponding to the prime hexagon, but \(\reg(\alpha_{15}\alpha_{12}\alpha_{46})=0\) because the initial diagram on the heptagon has a residual chord. The example shows a general fact: The algebra relations can only reduce the number of residual chords, not increase them.

\subsubsection{The residual weight filtration}\label{sec:resfiltration}

Define a filtration \(R^*\) on \(H^*(M_0(n))\) by declaring \(R^rH^*(M_0(n))\) to be the space of forms that have nonzero residue only to boundary strata of the dihedral moduli space that are of codimension \(\leq r\). Equivalently, it is spanned by gravity chord diagrams with \(\leq r\) residual chords. We shall call this \emph{the residual weight filtration}.

\subsubsection{The associated graded cooperad}\label{sec:regcoop}

Note that the associated graded \(\gr{}H^*(M_0)\) and \(H^*(M_0)\) are isomorphic collections. However, the associated graded has an induced cooperad structure, that differs. We shall denote the cocomposition \(\Regr\). Thus, e.g.,
\[
\Regr_{45}
\begin{tikzpicture}[baseline={([yshift=-.5ex]current bounding box.center)}]
\node[regular polygon, shape border rotate=360/14, regular polygon sides=7, minimum size=1cm, draw] at (5*2,0) (A) {};
\foreach \i in {1,...,7} {
    \coordinate (ci) at (A.corner \i);
    \coordinate (si) at (A.side \i);
    \node at (ci) {};
    %\path (A.center) -- (ci) node[pos=1.2] {};
    %\path (A.center) -- (si) node[pos=1.3] {\scriptsize{\i}}; 
 }
 \path[thick] (A.corner 1) edge (A.corner 3)
              (A.corner 1) edge (A.corner 6)
              (A.corner 4) edge (A.corner 7);
\end{tikzpicture}
 = 
\begin{tikzpicture}[baseline={([yshift=-.5ex]current bounding box.center)}]
\node[regular polygon, shape border rotate=360/6, regular polygon sides=6, minimum size=1cm, draw] at (5*2,0) (A) {};
\foreach \i in {1,...,6} {
    \coordinate (ci) at (A.corner \i);
    \coordinate (si) at (A.side \i);
    \node at (ci) {};
    %\path (A.center) -- (ci) node[pos=1.2] {};
    %\path (A.center) -- (si) node[pos=1.3] {\scriptsize{\i}}; 
 }
 \path[thick] (A.corner 1) edge (A.corner 3)
              (A.corner 1) edge (A.corner 5)
              (A.corner 3) edge (A.corner 6);
\end{tikzpicture} 
\otimes\!\!
\begin{tikzpicture}[baseline={([yshift=-.5ex]current bounding box.center)}]
\node[regular polygon, shape border rotate=360/6, regular polygon sides=3, minimum size=.7cm, draw] at (5*2,0) (A) {};
\foreach \i in {1,...,3} {
    \coordinate (ci) at (A.corner \i);
    \coordinate (si) at (A.side \i);
    \node at (ci) {};
    %\path (A.center) -- (ci) node[pos=1.2] {};
    %\path (A.center) -- (si) node[pos=1.5] {\scriptsize{\i}}; 
 }
\end{tikzpicture}.
\]
There is an isomorphism \((\gr{0} H^*(M_0), \Regr) \cong H^*(M^{\delta})\), as nonsymmetric cooperads. The general cocomposition \(\Regr_c G\) of the associated graded can be described by the following rule. The residual chords of \(G\) define a tesselation, and on each face we have a prime diagram. If \(c\) is equal to one of the residual chords or if it crosses a chord in \(G\), then \(\Regr_cG\) is zero. Otherwise it divides one of the prime diagrams on a face into two. Use the relations \emph{on those two faces}. The associated graded is, as an algebra, a sum of tensor powers of the algebra \(H^*(M^{\delta})\).

We deduce that we have a split inclusion
\[
  H^*(M^{\delta}) \hookrightarrow \gr{} H^*(M_0) \xrightarrow{\reg} H^*(M^{\delta})
\]
of nonsymmetric cooperads.

\subsection{The dihedral KZ connection}

\subsubsection{Dihedral braids}

Define the \textit{dihedral braid Lie algebra} on the set \([n]\), to be denoted \(\mathfrak{d}(n)\), to be the Lie algebra generated by variables \(\delta_{ij}\), \(i,j\in [n]\), modulo the relations that \(\delta_{ji}=\delta_{ij}\) for all indices, \(\delta_{ij}=0\) unless \(\{i,j\}\in \chi_1(n)\) is a chord, and
\[
  \bigl[\delta_{i\,j}+\delta_{i+1\,j-1}-\delta_{i+1\,j}-\delta_{i\,j-1},
  \delta_{k\,l}+\delta_{k+1\,l-1}-\delta_{k+1\,l}-\delta_{k\,l-1}\bigr]=0
\]
for all quadruples of indices in \([n]\) such that \(\#\{i,j,k,l\}=4\). Since the relations are quadratic we can and shall consider the dihedral braid Lie algebra as completed with respect to the number of brackets.

\subsubsection{Relation to ordinary braids}\label{sec:braids}

The \emph{spherical braid Lie algebra} is the Malcev Lie algebra \(\mathfrak{p}(n)\) generated by \(p_{ij}\) for \(1\leq i,j\leq n+1\) and relations \([p_{ij}, p_{kl}]=0\) if \(\#\{i,j,k,l\}=4\), and the linear relations \(p_{ij}=p_{ji}\), \(p_{ii}=0\) and \(\sum_{k=1}^{n+1} p_{ik}=0\) for all \(i\). Brown proved in \cite{Brown09} that setting
\[
  p_{ij} = \delta_{i\,j}+\delta_{i+1\,j-1}-\delta_{i+1\,j}-\delta_{i\,j-1}
\]
is an isomorphism of Lie algebras \(\mathfrak{d}(n) \cong \mathfrak{p}(n)\), i.e., the dihedral Lie algebra is just a different presentation (without linear relations!) of the spherical braid Lie algebra. The inverse is (assume \(i<j\), w.l.o.g.)
\[
  \delta_{ij} = \sum_{i \leq r < s\leq j} p_{rs}.
\]
The spherical braid Lie algebra is well-known to be isomorphic (over the rational numbers) to the Malcev Lie algebra of the mapping class group \(\pi_1(M_0(n))\). Since \(M_0(n)\) is a rational \(K(\pi, 1)\), this means that the Chevalley-Eilenberg cohomology \(H^*(\mathfrak{d}(n))\) is isomorphic to the cohomology of \(M_0(n)\).

Let \(\mathfrak{t}(n)\) be the usual braid Lie algebra, also known as the Drinfeld-Kohno Lie algebra. It has generators \(t_{ij}\) for \(1\leq i,j\leq n\) and relations \([t_{ij}, t_{kl}]=0\) for all quadruples of distinct indices, \([t_{ij},t_{ik}+t_{jk}]=0\) for all triples of distinct indices, and linear relations \(t_{ij}=t_{ji}\) and \(t_{ii}=0\). Its cohomology \(H^*(\mathfrak{t}(n))\) is isomorphic to the cohomology of the moduli space of configurations of \(n\) distinct points in the plane. The relations \(p_{in}=-\sum_{k=1}^n p_{ik}\) imply that the Lie algebra \(\mathfrak{p}(n)\), hence also \(\mathfrak{d}(n)\), is isomorphic to the quotient of \(\mathfrak{t}(n)\) by the additional relation \(2\sum_{1\leq i<j\leq n} t_{ij}=0\).

\subsubsection{Braids are an operad}

First, note that the isomorphism \(\mathfrak{d}(n)\cong\mathfrak{p}(n)\) means that the dihedral braid Lie algebra has an \(\bS_n\)-action, given by \(\sigma\cdot p_{ij} = p_{\sigma(i)\sigma(j)}\). Grafting trees can equivalently be though of as gluing polygons, giving functions
\[
  f\sqcup g : \chi_1(n/I)\sqcup\chi_1(I) \to \chi_1(n)
\]
that reinterpret a chord on a smaller polygon as a chord on the glued polygon. Define
\[
  \circ_i: \mathfrak{d}(n/I)\oplus \mathfrak{d}(I) \to \mathfrak{d}(n)
\]
by \(\delta_a\oplus \delta_b \mapsto \delta_{f(a)}+\delta_{g(b)}\). These compositions are obviously associative (because gluing polygons is), and defining relations are mapped to defining relations (only the indices change).

\subsubsection{The KZ connection}\label{sec:KZ}

The \emph{Knizhnik-Zamolodchikov connection} is usually given as the flat connection one-form 
\[
  \sum_{1\leq i < j\leq n+1} \omega_{ij}p_{ij}
\]
on \(M_0(n)\), where \(\omega_{ij} =\dr\log(z_i - z_j)\). Using the dihedral Lie algebra one can instead write it
\[
  \alpha = \sum_{c\in\chi_1(n)} \alpha_c\delta_c.
\]
Brown proves in \cite{Brown09} that it is compatible with regularised restriction to boundary strata. Equivalently, it is a morphism
\[
  \alpha: C^*(\mathfrak{d}) \to H^*(M_0)
\]
of cooperads of differential graded commutative algebras from the cooperad of Chevalley-Eilenberg cochains. It is in fact a quasi-isomorphism of cooperads, since \(H^*(\mathfrak{p})=H^*(M_0)\), as remarked in \ref{sec:braids}.

\subsubsection{Logarithmic explanation}

We have not checked the details but it seems probable that the operadic structure on dihedral braids can be deduced from the fact demonstrated in \ref{sec:logstruct}, that the dihedral moduli spaces are an operad of logarithmic varieties. Namely, to any logarithmic scheme \((X,\rM_X)/ \Spec(\Q,\N)\) with a rational point \(x\in X(\Q)\) one can associate a logarithmic de Rham fundamental group \(\pi_1^{\mathsf{dR}}(X,\rM_x,x)\), via a theory of nilpotent integrable logarithmic connections.\cite{Shiho00} If \(X\) is a smooth, proper scheme over \(\Q\) and the logarithmic structure \(\rM_X\) is defined by a normal crossings divisor \(D\subset X\), then the logarithmic de Rham fundamental groupoid of \((X,\rM_X)\) is isomorphic to the usual de Rham fundamental groupoid of \(X\setminus D\). Thus, although the dihedral braid Lie algebra is most immediately related to the cohomology of the open moduli space it is, in our context, better to think of it as related to the logarithmic de Rham cohomology of the dihedral moduli space, and to think of the Knizhnik-Zamolodchikov connection as a ``universal'' logarithmic connection on the dihedral moduli space.

\subsubsection{Regularised KZ}\label{sec:regKZ}

We call the composite morphism
\[
  \alpha_{\reg} :  C^*(\mathfrak{d}) \xrightarrow{\alpha} H^*(M_0)\cong \gr{} H^*(M_0) 
  \xrightarrow{\reg} H^*(M^{\delta})
\]
the \emph{regularised Knizhnik-Zamolodchikov connection}. Note that it is niether a morphism of cooperads, nor a connection (it is not defined by a one-form). It will nevertheless prove useful because, as we will later show, it is close enough to being a morphism of cooperads that it will allow us to relate the Hochschild cohomology of \(C_*(\mathfrak{d})\) to that of \(H_*(M^{\delta})\).

\section{GRT and Hochschild cohomology}

\subsection{The Grothendieck-Teichm\"uller Lie algebra}

The standard definition of \(\mathfrak{grt}_1\), the \emph{Grothendieck-Techm\"uller Lie algebra}, is that it is the set of all formal Lie series \(\psi\in\Lie\bbrackets{x,y}\) in two variables satisfying the \emph{pentagon equation}
\begin{align*}
  \psi(t_{12},t_{23}+t_{24}) &+ \psi(t_{13}+t_{23},t_{34}) =\\
   &\psi(t_{23},t_{34}) + \psi(t_{12}+t_{13},t_{24}+t_{34}) + \psi(t_{12},t_{23})
\end{align*}
inside the braid Lie algebra \(\mathfrak{t}(4)\) (cf.~\ref{sec:braids}) and the two symmetry equations
\begin{align*}
  \psi(x,y) + \psi(y,-x-y) + \psi(-x-y,x) &= 0, \\
  \psi(x,y) + \psi(y,x) &= 0.
\end{align*}
The Lie bracket is not the free bracket on \(\Lie\bbrackets{x,y}\), but the so-called Ihara bracket:
\[
  [\psi, \phi]_{\text{Ihara}} = D_{\psi}(\phi) - D_{\phi}(\psi) + [\psi,\phi],
\]
where \(D_{\xi}\) is the derivation defined by \(D_{\xi}(x) = 0\), \(D_{\xi}(y) = [y,\xi]\). However, we make no claims about compatibility with the Ihara bracket. It plays no role in what follows.

We now mention the following theorem by Furusho \cite{Furusho07}, with a later, independent proof by Willwacher \cite{Willwacher15}:

\begin{thm}\label{thm:furusho}
The space of formal Lie series in two variables that satisfy the pentagon equation is spanned by\, \([x,y]\) and\, \(\mathfrak{grt}_1\).
\end{thm}

\noindent
Thus, the two defining symmetry equations are (nearly) superfluous. The Grothendieck-Teichm\"uller Lie algebra is equivalently defined as the set of Lie series of depth at least two satisfying the pentagon equation.

\subsubsection{Definiton as Hochschild cohomology}\label{sec:dihcom}

The collection of dihedral braid Lie algebras is  an infinitesimal \(\As\)-bimodule, hence also its degree-shift \(s\mathfrak{d}\). The right actions
\[
  \partial_i = (-)\circ_i m_2 : s\mathfrak{d}(n) \to s\mathfrak{d}(n+1), \;\;\; 1\leq i\leq n
\]
of the binary generator \(m_2\in\As(2)\) are given by gluing a triangle to side \(i\) of the regular \((n+1)\)-gon, and reindexing the chords accordingly. Similarly, the left actions \(\partial_{n+1} = m_2\circ_1(-)\) and \(\partial_0 = m_2\circ_1(-)\) are given by gluing the polygon to side \(1\) or side \(2\) of a triangle, and reindexing.

\begin{prop}\label{prop:grtharr}
\(H^0(\As,s\mathfrak{d}) = \Q\oplus\mathfrak{grt}_1\).
\end{prop}

\begin{proof}
The degree zero Harrison cohomology of the dihedral braid Lie algebras is for degree reasons just the set of degree zero cocycles, which is to say the set of
\(
  \psi \in  \mathfrak{d}(3)
\)
satisfying the cocycle equation
\[
\partial_0\psi - \partial_1\psi + \partial_2\psi - \partial_3\psi + \partial_4\psi = 0.
\]
First, there are only two chords \(x=\delta_{12}\) and \(y=\delta_{23}\) on a square, so \(\mathfrak{d}(3)\) is the free algebra \(\Lie\bbrackets{x,y}\). The cocycle equation is
\[
  \psi(\delta_{23},\delta_{34}) - \psi(\delta_{13},\delta_{34}) + \psi(\delta_{13},\delta_{24}) - \psi(\delta_{12},\delta_{24}) + \psi(\delta_{12},\delta_{23}) = 0,
\]
and it is of course equivalent to the pentagon equation. Using the coordinate change of \ref{sec:braids} it can be written
\begin{align*}
  \psi(p_{12}&+p_{23}+p_{13},p_{34}) + \psi(p_{12},p_{23}+p_{34}+p_{24}) =\\
   &\psi(p_{23},p_{34}) + \psi(p_{12}+p_{23}+p_{13},p_{23}+p_{34}+p_{24}) + \psi(p_{12},p_{23}).
\end{align*}
Using that \(p_{ij}\) and \(p_{kl}\) commute if all four indices are different and that any generator commutes with itself then yields the standard form of the pentagon equation. The result follows from \ref{thm:furusho}.
\end{proof}

\noindent
The analogous statement with the Lie algebra of dihedral braids replaced by the usual Drinfeld-Kohno Lie algebra \(\mathfrak{t}\) is presented in \cite{Willwacher15}.  In detail, \(\mathfrak{grt}_1 = H^0(\Com, s\mathfrak{t})\). That is to say, the GrothendieckTechm\"uller Lie algebra consists of elements that satisfy the pentagon equation and have degree \(1\) in the Hodge-type decomposition 
\[
  \mathfrak{t}(3) =  e^1_3\ \mathfrak{t}(3) \oplus e^2_3\ \mathfrak{t}(3)\oplus e^3_3\ \mathfrak{t}(3).
\]
The equation \(e^2_3\psi = 0\)  is \(\psi(x,y) + \psi(y,x) = 0\). The equation \(e^3_3\psi = 0\) is analogously equivalent to
\[
 \psi(x,y) + \psi(y,-x-y) + \psi(-x-y,x) = 0,
\]
the second defining symmetry equation.

\subsection{Hochschild cohomology of moduli spaces}

The goal of this section is to show that \(H^0(\As, H_*(M^{\delta})) \cong H^0(\Ass, H_*(M_0))\).

\subsubsection{Dual Hochschild complex}\label{sec:dualhoch}

For any infinitesimal \(\As\)-bimodule \(N^{\vee}\),
\[
  C(\As, N^{\vee}) = \biggl( \bigoplus_{n\geq 2} \susp{n-2} N(n) \biggr)^{\!\vee}.
\]
Call the complex in parantheses on the right the \emph{dual Hochschild cochain complex} of \(N\). The differential is defined by ``coactions''; in the cases when \(N\) equals \(H^*(M_0)\) or \(H^*(M^{\delta})\) the differential is the sum with alternating signs of all ways to ``cut a corner'' (as cutting a corner is dual to gluing a triangle). In detail, the differential is in both these cases
\[
  \dr = - \sum_{i=0}^n(-1)^i \dr_i : N(n) \to N(n),
\]
where 
\[
  \dr_0 = \Reg_{2n},\ \dr_n = \Reg_{1\, n-1}\ \text{and}\ \dr_i = \Reg_{i\, i+1}
\]
for \(i=1,\dots, n-1\).

\subsubsection{Diagrammatic pullbacks}

Both the open and Brown's dihedral moduli spaces have point-forgetting projections, by which one can pull-back differential forms. Let us focus on the open moduli space, for now, and denote the projection that forgets the \(i\)th marked point \(z_i\) by \(\pi_i: M_0(n) \to M_0(n-1)\). The pullback
\[
  \pi^*_i : H^*(M_0(n-1)) \to H^*(M_0(n))
\]
maps a chord-form \(\alpha_c\) on the \(n\)-gon to the sum \(\pi^*_i\alpha_c = \sum \alpha_b\) over all chords \(b\) on the \((n+1)\)-gon that become identified with \(c\) when the \(i\)th side is contracted. For example:
\[
\pi^*_1\,
\begin{tikzpicture}[font=\scriptsize, baseline={([yshift=0ex]current bounding box.center)}]
\node[regular polygon, shape border rotate=36, regular polygon sides=5, minimum size=1cm, draw] at (0,0) (A) {};
 \path[thick] (A.corner 1) edge node[auto] {} (A.corner 4)
              (A.corner 2) edge node[auto] {} (A.corner 5);
\end{tikzpicture}
\,=\,
\begin{tikzpicture}[font=\scriptsize, baseline={([yshift=-.2ex]current bounding box.center)}]
\node[regular polygon, shape border rotate=60, regular polygon sides=6, minimum size=1cm, draw] at (0,0) (A) {};
 \path[thick] (A.corner 3) edge node[auto] {} (A.corner 6)
              (A.corner 1) edge node[auto] {} (A.corner 5);
\end{tikzpicture}
\,+\,
\begin{tikzpicture}[font=\scriptsize, baseline={([yshift=-.2ex]current bounding box.center)}]
\node[regular polygon, shape border rotate=60, regular polygon sides=6, minimum size=1cm, draw] at (0,0) (A) {};
 \path[thick] (A.corner 3) edge node[auto] {} (A.corner 6)
              (A.corner 2) edge node[auto] {} (A.corner 5);
\end{tikzpicture}\,.
\]

\subsubsection{Residual chords are exact}

We will prove that the associated graded for the filtration by residual weight (cf.~\ref{sec:resfiltration}) on the dual Hochschild cochain complex of \(H^*(M_0)\) only has cohomology in residual weight zero.

\begin{lem}\label{lem:triviality}
\(\dr_i\circ\, \pi^*_1 = 0\) restricted to\, \(H^{n-3}(M^{\delta}(n-1))\)\, for \(i=2,\dots, n-1\).
\end{lem}

\begin{proof}
Recall \(\dr_i = \Reg_{i\,i+1}\). When interpreted diagrammatically the kernel of \(\dr_i\) consists of all diagrams that either contain the chord \(c_{i\,i+1}\) or a chord that crosses it. The assumption \(2\leq i\leq n-2\) means that \(1\) is not one of the two sides enclosed by the chord \(c_{i\,i+1}\). Moreover, the pullback \(\pi\mathrlap{^*}_1G\) of a chord diagram \(G\) is a sum of diagrams that all contain a chord crossing \(c_{i\,i+1}\) if and only if \(G\) contains a chord that crosses \(c_{i-1\,i}\).

All prime gravity chord diagrams of top degree on an \(n\)-gon (except \(\triang\)) contain a chord crossing \(c_{2\,n-1}\). This is perhaps most evident recalling the correspondence with prime Lie words given in \ref{sec:liecorr}. Any such word must contain a bracketing \([1,\dots]\), hence a chord of the form \(c_{1k}\).

If now \(\sigma\in\Z/n\Z\) is a cyclic rotation, then the set \(\{\sigma(P)\}\) of rotated prime gravity chord diagrams must also be a basis. Equivalently, we defined the notion of being a gravity chord diagram by distinguishing the ``top'' side, i.e., the side \(n+1\), but could just as well have distinguished another side. Thus, we can construct a basis of \(H^{n-2}(M^{\delta}(n-1))\) consisting of diagrams that all contain a chord crossing \(c_{i\,i+1}\).
\end{proof}

\begin{cor}\label{cor:onto}
The linear map
\[
  \sum_{i=1}^{n-1} (-1)^i\dr_i : H^{n-3}(M^{\delta}(n)) \to H^{n-3}(M^{\delta}(n-1))
\]
is onto.
\end{cor}

\begin{proof}
It follows from lemma \ref{lem:triviality} that \(\sum_{i=1}^{n-1} (-1)^i \dr_i\circ\,\pi^*_1 = -\dr_1\circ \,\pi^*_1\). Then note that we have \(\dr_1\circ\, \pi^*_1=\id\).
\end{proof}

Consider the filtration by residual weight on the dual Hochschild cochain complex of the open moduli space, and its associated graded. We may identify it with the dual Hochschild cochain complex of the associated graded cooperad \(\gr{} H^*(M_0)\), introduced in \ref{sec:regcoop}.

\begin{lem}\label{lem:residexact}
The associated graded for the filtration by residual weight of the dual Hochschild cochain complex of \(H^*(M_0)\) only has cohomology in weight zero.
\end{lem}

\begin{proof}
Because of theorem \ref{thm:almpetersen} a gravity chord diagram is equivalent to a tesselation (the residual chords) of the polygon into a number of smaller polygons, and on each face a prime gravity chord diagram. Thus, consider a diagram of the form \(G\wedge c\wedge P\) with \(c\) a residual chord and \(G\) and \(P\) top degree gravity chord diagrams and \(P\) prime. To prove the proposition we must show that such a diagram is exact for the ``corner cutting'' differential of the associated graded of the dual Hochschild cochain complex.

First, imagine adding one side to the face where \(P\) lives and putting some prime diagram \(P'\) on that face, viz:
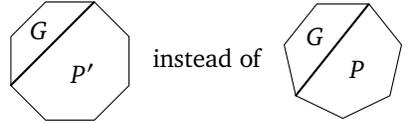
\[
\begin{tikzpicture}[baseline={([yshift=-.5ex]current bounding box.center)}]
  \node[regular polygon, shape border rotate=45,
        regular polygon sides=8, minimum size=1.7cm, draw] at (0,0) (A) {};
  \coordinate (g) at (A.side 1);
  \coordinate (p) at (A.side 5);
  \path (A.center) -- (g) node[pos=.73] {$G$};
  \path (A.center) -- (p) node[pos=.3] {$P'$}; 
  \path[thick] (A.corner 3) edge node[auto] {} (A.corner 8);
\end{tikzpicture}\;\;\;
\text{instead of}\;\;\;
\begin{tikzpicture}[baseline={([yshift=-.5ex]current bounding box.center)}]
  \node[regular polygon, shape border rotate=360/14,
        regular polygon sides=7, minimum size=1.6cm, draw] at (0,0) (A) {};
  \coordinate (g) at (A.side 1);
  \coordinate (p) at (A.corner 5);
  \path (A.center) -- (g) node[pos=.64] {$G$};
  \path (A.center) -- (p) node[pos=.3] {$P$}; 
  \path[thick] (A.corner 3) edge node[auto] {} (A.corner 7);
\end{tikzpicture}\,.
\]
Since \(G\) is of top degree we can not cut any corner on that face. On the polygon enclosing \(P'\) we can (a priori) cut any corner except the two that meet the residual chord, since those are corners of \(P'\) and not of the big diagram. In other words, the corner cutting differential defined by the graded regularised restrictions \(\Regr_{ij}\) will act schematically as
\[
  \dr(G\wedge c\wedge {\textstyle{P'}}) = \pm G \wedge c\wedge \sum_{i=1}^{n-1}(-1)^i\dr_i {\textstyle{P'}}
\]
if the face of \(P'\) is an \((n+1)\)-gon. Thus, by corollary \ref{cor:onto}, the initial diagram \(G\wedge c\wedge P\) is exact.
\end{proof}

\subsubsection{Corollary: a map from GRT}

Let \(D(\As,-)\) denote the dual Hochschild cochain complex. By lemma \ref{lem:residexact} and \ref{sec:regcoop} we have a split inclusion
\[
 D(\As, H^*(M^{\delta})) \hookrightarrow D(\As,\gr{} H^*(M_0)) \xrightarrow{\reg}  D(\As, H^*(M^{\delta})).
\]
with \(\reg\) inducing an isomorphism on degree zero cohomology.

\begin{thm}\label{thm:isohoch}
\(H^0(\As,H_*(M^{\delta}))\cong H^0(\Ass, H_*(M_0))\).
\end{thm}

\begin{proof}
Filter \(D(\As, H^*(M_0))\) by residual weight and consider the induced spectral sequence. We have \(E_0^{p,q} = 0\) if \(p+q>0\) for degreee reasons. Moreover, \(p\geq 0\) because a chord diagram cannot have less than zero residual chords. By lemma \ref{lem:residexact}, \(E_1^{p,-p}=0\) unless \(p=0\), and then we have \(E_1^{0,0} = H^0(D(\As, H^*(M^{\delta})))\). By what has been said the spectral sequence abuts in degree zero already at the first page.
\end{proof}

Equivalently, the degree one cohomologies of the deformation complexes of the respective morphism from the associative operad are isomorphic. We can now deduce:

\begin{thm}\label{thm:grtinj}
The regularised KZ connection (see section \ref{sec:regKZ}) induces a map
\[
\alpha_{\reg}^{\vee} : \mathfrak{grt}_1 \rightarrow H^0(\As,H_*(M^{\delta})).
\]
from the Grothendieck-Teichm\"uller Lie algebra to the operadic Hochschild cohomology of the homology operad of Brown's dihedral moduli spaces.
\end{thm}

\begin{proof}
By \ref{prop:grtharr}, the map \(\psi(x,y)\mapsto \psi(\delta_{12},\delta_{23})\) gives an isomorphism
\[
 \Q \oplus \mathfrak{grt}_1 \to H^0(\As, s\mathfrak{d}).
\]
We have a subcomplex inclusion \(C(\As, s\mathfrak{d})\to C(\As, C_*(\mathfrak{d}))\) and a quasi-isomorphism
\[
\alpha_{\reg} : D(\As, C^*(\mathfrak{d})) \to D(\As, H^*(M_0)).
\]
On taking cohomology this map can by \ref{thm:isohoch} be taken to land in the space of classes living on the dihedral moduli space. 
\end{proof}

\subsection{The zeta of two class}

Observe that
\begin{align*}
  H^2(M^{\delta}(4)) &= \, \Q\,
  \begin{tikzpicture}[font=\scriptsize, baseline={([yshift=-.5ex]current bounding box.center)}]
    \node[regular polygon, shape border rotate=36,
          regular polygon sides=5, minimum size=.75cm, draw] at (0,0) (A) {};
    \path[thick] (A.corner 1) edge node[auto] {} (A.corner 4)
                 (A.corner 2) edge node[auto] {} (A.corner 5);
  \end{tikzpicture}\,, \\
  H^2(M^{\delta}(5)) &= \,\Q\biggl\{\,
  \begin{tikzpicture}[font=\scriptsize, baseline={([yshift=-.6ex]current bounding box.center)}]
    \node[regular polygon, shape border rotate=60,
          regular polygon sides=6, minimum size=.75cm, draw] at (0,0) (A) {};
    \path[thick] (A.corner 2) edge node[auto] {} (A.corner 5)
                 (A.corner 3) edge node[auto] {} (A.corner 6);
  \end{tikzpicture}
  \,,\,
  \begin{tikzpicture}[font=\scriptsize, baseline={([yshift=-.85ex]current bounding box.center)}]
    \node[regular polygon, shape border rotate=60,
          regular polygon sides=6, minimum size=.75cm, draw] at (0,0) (A) {};
    \path[thick] (A.corner 1) edge node[auto] {} (A.corner 5)
                 (A.corner 3) edge node[auto] {} (A.corner 6);
  \end{tikzpicture}
  \,,\,
  \begin{tikzpicture}[font=\scriptsize, baseline={([yshift=-.85ex]current bounding box.center)}]
    \node[regular polygon, shape border rotate=60,
          regular polygon sides=6, minimum size=.75cm, draw] at (0,0) (A) {};
    \path[thick] (A.corner 1) edge node[auto] {} (A.corner 5)
                 (A.corner 2) edge node[auto] {} (A.corner 6);
  \end{tikzpicture}
  \,,\,
  \begin{tikzpicture}[font=\scriptsize, baseline={([yshift=-.85ex]current bounding box.center)}]
    \node[regular polygon, shape border rotate=60,
          regular polygon sides=6, minimum size=.75cm, draw] at (0,0) (A) {};
    \path[thick] (A.corner 1) edge node[auto] {} (A.corner 4)
                 (A.corner 2) edge node[auto] {} (A.corner 6);
  \end{tikzpicture}
  \,,\,
  \begin{tikzpicture}[font=\scriptsize, baseline={([yshift=-.6ex]current bounding box.center)}]
    \node[regular polygon, shape border rotate=60,
          regular polygon sides=6, minimum size=.75cm, draw] at (0,0) (A) {};
    \path[thick] (A.corner 1) edge node[auto] {} (A.corner 4)
                 (A.corner 2) edge node[auto] {} (A.corner 5);
  \end{tikzpicture}\,\biggr\}.
\end{align*}
All five differential forms displayed on the second line are easily verified to be cocycles for the corner cutting differential of the dual Hochschild cochain complex. In particular, the pentagon diagram shown on the first line is a nontrivial class in the degree zero Hochschild cohomology.

\subsubsection{Lie algebra interpretation}

\(\pent = \delta^\vee_{13}\wedge \delta^\vee_{24}\) is not a cocycle for the corner cutting differential when considered in the cooperad \(C^*(\mathfrak{d})\) of Chevalley-Eilenberg cochains of the dihedral Lie algebra. Instead
\[
  \dr\,
  \begin{tikzpicture}[font=\scriptsize, baseline={([yshift=-.5ex]current bounding box.center)}]
    \node[regular polygon, shape border rotate=36,
          regular polygon sides=5, minimum size=.6cm, draw] at (5*2,0) (A) {};
    \path[thick] (A.corner 1) edge node[auto] {} (A.corner 4)
                 (A.corner 2) edge node[auto] {} (A.corner 5);
  \end{tikzpicture}
  \,=\, -\,
  \begin{tikzpicture}[font=\scriptsize, baseline={([yshift=-.5ex]current bounding box.center)}]
    \node[regular polygon, shape border rotate=0,
          regular polygon sides=4, minimum size=.6cm, draw] at (5*2,0) (A) {};
    \path[thick] (A.corner 1) edge node[auto] {} (A.corner 3)
                 (A.corner 2) edge node[auto] {} (A.corner 4);
  \end{tikzpicture}
  \,=\,
  - \delta^\vee_{12}\wedge \delta^\vee_{23}
  \,=\, \dr_{CE}\bigl(-[\delta_{12},\delta_{23}]^{\vee}\bigr).
\]
Thus \((\dr + \dr_{CE})([\delta_{12},\delta_{23}]^{\vee} + \pent) = 0\). By repeating the arguments in the proof of proposition \ref{prop:grtharr}, one sees that \(H^0(\Ass, s\mathfrak{d})\) consists of all formal Lie series \(\psi\in\Lie\bbrackets{x,y}\) that satisfy the pentagon equation (neither of the two symmetry equations is enforced). We conclude that \(\psi = [x,y]\) is the representative of \(\pent\in H^0(\As,H_*(M^{\delta}))\).

It is easily seen that \([x,y]\) satisfies the pentagon equation but neither of the two symmetry equations of \(\mathfrak{grt}_1\). In particular, the pentagon class does not come from an element of the Grothendieck-Teichm\"uller Lie algebra (via the regularised Knizhnik-Zamolodchikov connection), so we conclude that we have a map
\[
  \Q\triang \oplus\Q\pent\oplus\mathfrak{grt}_1 \rightarrow H^0(\As,H_*(M^{\delta})).
\]
This is consistent with \ref{thm:furusho}. Our main goal in this paper is to show that this map is an isomorphism.

\subsubsection{Period integral interpretation}

For every top degree differential form \(\gamma\in H^{n-2}(M^{\delta}(n))\), one has a convergent period integral
\[
  \int_{0=z_1 < z_2 <\dots < z_n=1} \gamma.
\]
A remarkable theorem by Brown \cite{Brown09} says:

\begin{thm}\label{thm:brownperiods}
Any period integral as displayed above is a rational linear combination of multiple zeta values of weight \(n-2\) and all multiple zeta values of weight \(n-2\) appear as such an integral.
\end{thm}

\noindent
The period of \(\triang\) is \(1\) and the period of \(\pent\) is \(\zeta(2)\). A famous conjecture says that the indecomposable quotient of the algebra of multiple zeta values is isomorphic to the dual of the Grothendieck-Teichm\"uller Lie algebra, plus one element corresponding to zeta of two. How all this is related to our calculation of \(H^0(\As, H_*(M^{\delta}))\) is discussed at greater length in \ref{sec:QZ}.

\section{Intermission}

We will in the next section prove that \(\Q\triang \oplus\Q\pent\oplus\mathfrak{grt}_1 \rightarrow H^0(\As,H_*(M^{\delta}))\) is an isomorphism, but before doing so take a detour through preliminiaries that will feature in our argument. All of the results in this section are already known in all essential detail, scattered in the literaure, so we will allow ourselves to be brief.

\subsection{The Batalin-Vilkovisky operad}

A \emph{Batalin-Vilkovisky algebra} is a differential graded vector space with three algebraic operations: (i) a differential graded commutative product; (ii) a differential graded Lie bracket of degree \(-1\), and (iii) a nilsquare operator \(\Delta\) of degree \(-1\) such that \([g,h] = \Delta(gh) - \Delta(g)h -(-1)^{\vert g\vert} g\Delta h\). We denote the operad of Batalin-Vilkovisky algebras \(\mathsf{BV}\).

\subsubsection{Spherical ribbon braids}

Let \(\mathfrak{rb}(n)\) denote \emph{the spherical ribbon braid Lie algebra} on \(n+1\) strands. (Note the mismatch between arity and number of strands.) It is the Lie algebra of formal Lie series in the symbols \(b_{ij}\) and \(s_k\), \(1\leq i, j, k\leq n+1\), modulo the linear relations \(b_{ii}=0\), \(b_{ji}=b_{ij}\) and
\[
  2s_k + \sum_{i=1}^{n+1} b_{ik} = 0\;\;\;\text{for all}\ k=1,\dots, n+1,
\]
and the bracket relations
\begin{align*}
  [s_k\,,\,\text{anything}] &= 0\;\;\;\text{for all}\ k=1,\dots, n+1; \\
  [b_{ij}, b_{kl}] &= 0\;\;\;\text{if}\ \#\{i,j,k,l\}=4.
\end{align*}
Note that one may use the linear relations to eliminate all \(b_{i\,n+1}\) and \(s_{n+1}\).

The collection of spherical ribbon braid Lie algebras forms an operad of graded Lie algebras. Explicit formulas are somewhat involved but can be given a more transparent explanation by graphical means, something we turn to in the next subsection.

The operad of Batalin-Vilkovisky algebras, \(\mathsf{BV}\), is isomorphic to the operad \(H_*(\mathfrak{rb})\) of Chevalley-Eilenberg homologies of the spherical ribbon braids. The commutative product is given by the class \(1\in H_0(\mathfrak{rb}(2))\), the Lie bracket operation by \(b_{12}\in H_{-1}(\mathfrak{rb}(2))\) and the unary operation \(\Delta\) corresponds to the class \(s_1\in H_{-1}(\mathfrak{rb}(1))\).

\subsubsection{Programmatic remark}

The collection of spherical ribbon braid Lie algebras is a module for the commutative operad. Arguing as in \ref{sec:dihcom} quickly leads to the conclusion that \(\mathfrak{grt}_1 \subset H^0(\As,\susp{}\mathfrak{rb})\). In fact, \(\mathfrak{rb}\) is a symmetric operad, and we have \(\mathfrak{grt}_1 \subset H^0(\Com,\susp{}\mathfrak{rb})\).

\begin{prop}\label{prop:openHarr}
\(\Q^{\oplus 2}\oplus\mathfrak{grt}_1\subset H^0(\Com, \mathsf{BV}) \).
\end{prop} 

\begin{proof}
We first note that the infinitesimal left action of\, \(\Com\) on \(s\mathfrak{rb}\) is free. Verily, the action extends to an action of the unitary operad \(\mathsf{uCom}\). In more detail, we have projections \(\pi_i:\mathfrak{rb}(n) \to \mathfrak{rb}(n-1)\), analogous to projections that forget a marked point. These projections \(\pi_1\) and \(\pi_n\) are right inverses to the two left actions \(\partial_0\) and \(\partial_{n}\), so that they must be injective.

Note that we have a splitting 
\[
  C(\Com,C_*(\mathfrak{rb})) = C(\Com, C_0(\mathfrak{rb})) \oplus C(\Com, C_{< 0}(\mathfrak{rb})).
\]
The first summand has cohomology \(\Q\) for degree reasons. This class corresponds to the deformation that simply rescales the image of \(m_2\). Next, note \(C_{<0}(\mathfrak{rb}) = \hat{S}^{>0}(s\mathfrak{rb})\) as a collection of graded vector spaces. Filter by tensor degree, so that the associated graded is 
\[
  \bigoplus_{k\geq 1} C(\Com, S^k(s\mathfrak{rb})).
\]
The structure of infintesimal \(\Com\)-bimodule on \(S^k(s\mathfrak{rb})\), using that \(\Com\) is a Hopf operad. It now follows from the above remark that the infinitesimal module-action is free and the ``Tamarkin argument'' lemma \ref{lem:tamarkin} that \(C(\Com, S^k(s\mathfrak{rb}))\) is acyclic for \(k\geq 2\). Thus, the \(E_1\) page reduces to the subcomplex
\[
  C(\Com, s\mathfrak{rb}) \subset C(\Com, C_{<0}(\mathfrak{rb})),
\]
whence the cohomology is concentrated in this subcomplex. Repeating the arguments at the end of \ref{sec:dihcom} then proves the result.
\end{proof}

In the sections that follow we shall argue that the Grothendieck-Teichm\"uller Lie algebra is essentially the whole degree zero cohomology, even after passing to Hochschild instead of Harrison cohomology of \(\mathsf{BV}\). In more detail, we will prove that
\[
  H^0(\Ass,\mathsf{BV}) = 
  \Q\,
  \begin{tikzpicture}[baseline=-0.35ex,shorten >=0pt,auto,node distance=.8cm]
    \node[ve] (w1) {};
    \node[ve] (w2) [right of=w1] {};
  \end{tikzpicture}
  \,\oplus\,
  \Q\,
  \begin{tikzpicture}[baseline=-.1ex,shorten >=0pt,auto,node distance=.7cm]
    \node[ve] (w1) {};
    \node[ve] (w2) [right of=w1] {};
    \node[ve] (w3) [right of=w2] {};
    \node[vi] (z) [above of=w2] {};  
    \path[every node/.style={font=\sffamily\small}]
      (w1) edge (z)
      (w2) edge (z)
      (w3) edge (z); 
  \end{tikzpicture}
  \,\oplus \mathfrak{grt}_1.
\]
Here \(
  \begin{tikzpicture}[baseline=-0.35ex,shorten >=0pt,auto,node distance=.8cm]
    \node[ve] (w1) {};
    \node[ve] (w2) [right of=w1] {};
  \end{tikzpicture}
\) is the class of rescaling the product while \(
  \begin{tikzpicture}[baseline=-.1ex,shorten >=0pt,auto,node distance=.7cm]
    \node[ve] (w1) {};
    \node[ve] (w2) [right of=w1] {};
    \node[ve] (w3) [right of=w2] {};
    \node[vi] (z) [above of=w2] {};  
    \path[every node/.style={font=\sffamily\small}]
      (w1) edge (z)
      (w2) edge (z)
      (w3) edge (z); 
  \end{tikzpicture}
\) is a sibling of the ``zeta of two class'' \(\pent\). We then prove our main theorem by constructing an isomorphism \(H^0(\As,H_*(M^{\delta}))\cong H^0(\As,\mathsf{BV})\).

\subsubsection{Graphs with white vertices}

Define 
\[
  \BVGra = C_*\bigl(\mathfrak{rb}/[\mathfrak{rb},\mathfrak{rb}]\bigr)
\]
to be the collection of Chevalley-Eilenberg complexes on the abelianisation of the ribbon braid Lie algebra(s). Thus, elements are formal series in monomials of the generators \(b_{ij}\) and \(s_k\). We represent such monomials graphically, by drawing a number of white vertices equal to the arity, for each \(b_{ij}\) an edge between vertices \(i\) and \(j\), and for each \(k\) a ``tadpole'' (i.e., a loop edge) at vertex \(k\). In doing this we make use of the linear relations to eliminate all \(b_{i\,n+1}\) and \(s_{n+1}\), so that \(n\) (the arity) vertices suffice.

This collection has an operad structure, by the following rule. Let \(\Gamma\) and \(\Gamma'\) be graphs (monomials). To form \(\Gamma\circ_i\Gamma'\), first delete the vertex \(i\) of \(\Gamma\) and then sum over all ways to reconnect the edges of \(\Gamma\) that previously connected to \(i\) to some vertex of \(\Gamma'\). This prosaic definition fixes everything except the signs. To fix the signs one must describe how to consistently regard the edges of the constituent graphs as ordered (up to an even permutation). We shall omit this detail since we will not need the signs to carry through our arguments. Here is an example:
\begin{align*}
\begin{tikzpicture}[baseline=0.2ex,shorten >=0pt,auto,node distance=1.5cm]
 \node[ve] (w1) {};
 \node[ve] (w2) [right of=w1] {};
 \node[ve] (w3) [right of=w2] {}; 
   \path[every node/.style={font=\sffamily\small}]
    (w1) edge[out=55, in=125] (w2)
    (w2) edge[out=125, in=55, loop] (w2); 
\end{tikzpicture}
\;\circ_2\;
\begin{tikzpicture}[baseline=0.2ex,shorten >=0pt,auto,node distance=1.5cm]
 \node[ve] (w1) {};
 \node[ve] (w2) [right of=w1] {};
\end{tikzpicture}
\; &= \; 
\begin{tikzpicture}[baseline=0.2ex,shorten >=0pt,auto,node distance=1.5cm]
 \node[ve] (w1) {};
 \node[ve] (w2) [right of=w1] {};
 \node[ve] (w3) [right of=w2] {}; 
 \node[ve] (w4) [right of=w3] {};
   \path[every node/.style={font=\sffamily\small}]
    (w1) edge[out=55, in=125] (w2)
    (w2) edge[out=125, in=55, loop] (w2); 
\end{tikzpicture}
\; + \;
\begin{tikzpicture}[baseline=0.2ex,shorten >=0pt,auto,node distance=1.5cm]
 \node[ve] (w1) {};
 \node[ve] (w2) [right of=w1] {};
 \node[ve] (w3) [right of=w2] {}; 
 \node[ve] (w4) [right of=w3] {};
   \path[every node/.style={font=\sffamily\small}]
    (w1) edge[out=55, in=125] (w2)
    (w2) edge[out=55, in=125] (w3); 
\end{tikzpicture}
\; + \;
\begin{tikzpicture}[baseline=0.2ex,shorten >=0pt,auto,node distance=1.5cm]
 \node[ve] (w1) {};
 \node[ve] (w2) [right of=w1] {};
 \node[ve] (w3) [right of=w2] {}; 
 \node[ve] (w4) [right of=w3] {};
   \path[every node/.style={font=\sffamily\small}]
    (w1) edge[out=55, in=125] (w2)
    (w3) edge[out=125, in=55, loop] (w3); 
\end{tikzpicture} \\
 &+ \; 
\begin{tikzpicture}[baseline=0.2ex,shorten >=0pt,auto,node distance=1.5cm]
 \node[ve] (w1) {};
 \node[ve] (w2) [right of=w1] {};
 \node[ve] (w3) [right of=w2] {}; 
 \node[ve] (w4) [right of=w3] {};
   \path[every node/.style={font=\sffamily\small}]
    (w1) edge[out=55, in=125] (w3)
    (w2) edge[out=125, in=55, loop] (w2); 
\end{tikzpicture}
\; + \;
\begin{tikzpicture}[baseline=0.2ex,shorten >=0pt,auto,node distance=1.5cm]
 \node[ve] (w1) {};
 \node[ve] (w2) [right of=w1] {};
 \node[ve] (w3) [right of=w2] {}; 
 \node[ve] (w4) [right of=w3] {};
   \path[every node/.style={font=\sffamily\small}]
    (w1) edge[out=55, in=125] (w3)
    (w2) edge[out=55, in=125] (w3); 
\end{tikzpicture}
\; + \;
\begin{tikzpicture}[baseline=0.2ex,shorten >=0pt,auto,node distance=1.5cm]
 \node[ve] (w1) {};
 \node[ve] (w2) [right of=w1] {};
 \node[ve] (w3) [right of=w2] {}; 
 \node[ve] (w4) [right of=w3] {};
   \path[every node/.style={font=\sffamily\small}]
    (w1) edge[out=55, in=125] (w3)
    (w3) edge[out=125, in=55, loop] (w3); 
\end{tikzpicture}\;.
\end{align*}
We remark that this rule defines the operad structure on spherical ribbon braids.

\subsubsection{Graphs with black and white vertices}

Set
\[
  \mathsf{Tw}\,\BVGra(n) = \prod_{k\geq 0} \susp{-2k}\BVGra\bigl(k+n\bigr)_{\bS_k}.
\]
We depict elements as formal series of graphs with some number \(k\) of black unlabelled vertices and some number \(n\) (the arity) of white labelled vertices. Note that the degree is
\(
  2\#(\text{black vertices}) - \#(\text{edges}),
\)
counting tadpoles as edges. Note also that tadpoles can occur at both black and white vertices. This collection is again an operad, by extending the rule for composition of graphs with only white vertices, e.g.,
\[
  \begin{tikzpicture}[baseline=0.2ex,shorten >=0pt,auto,node distance=1.5cm]
    \node[ve] (w1) {};
    \node[ve] (w2) [right of=w1] {};
    \node[ve] (w3) [right of=w2] {};
    \node[vi] (b1) [above of=w2] {};
    \path[every node/.style={font=\sffamily\small}]
      (b1) edge (w1)
      (b1) edge (w2)
      (b1) edge (w3); 
  \end{tikzpicture}
  \;\circ_2\,
  \begin{tikzpicture}[baseline=0.2ex,shorten >=0pt,auto,node distance=1.5cm]
    \node[ve] (w1) {};
    \node[ve] (w2) [right of=w1] {};
    \path[every node/.style={font=\sffamily\small}]
      (w1) edge[out=55, in=125] (w2);
  \end{tikzpicture}
  \;=\; 
  \begin{tikzpicture}[baseline=0.2ex,shorten >=0pt,auto,node distance=1.5cm]
    \node[ve] (w1) {};
    \node[ve] (w2) [right of=w1] {};
    \node[ve] (w3) [right of=w2] {}; 
    \node[ve] (w4) [right of=w3] {};
    \node[vi] (b1) [above of=w2] {};
    \path[every node/.style={font=\sffamily\small}]
      (b1) edge (w1)
      (b1) edge (w2)
      (b1) edge (w4)
      (w2) edge[out=55, in=125] (w3); 
  \end{tikzpicture}
  \;+\;
  \begin{tikzpicture}[baseline=0.2ex,shorten >=0pt,auto,node distance=1.5cm]
    \node[ve] (w1) {};
    \node[ve] (w2) [right of=w1] {};
    \node[ve] (w3) [right of=w2] {}; 
    \node[ve] (w4) [right of=w3] {};
    \node[vi] (b1) [above of=w3] {};
    \path[every node/.style={font=\sffamily\small}]
      (b1) edge (w1)
      (b1) edge (w3)
      (b1) edge (w4)
      (w2) edge[out=55, in=125] (w3); 
  \end{tikzpicture}\,.
\]
The deformation complex
\[
  \Def\bigl({\textstyle\Sigma^{-1}}\Lie \xrightarrow{0} \BVGra\bigr) = \Map_{\bS}\bigl({\textstyle\Sigma^{-2}}\mathsf{coCom}, \BVGra\bigr)
\]
consists of formal series of graphs with symmetrised vertices; hence in the present context are naturally thought of as graphs with only black vertices. The graph \(
  \begin{tikzpicture}[baseline=-0.35ex,shorten >=0pt,auto,node distance=1cm]
    \node[vi] (w1) {};
    \node[vi] (w2) [right of=w1] {};
    \path[every node/.style={font=\sffamily\small}]
      (w1) edge[out=55, in=125] (w2);
  \end{tikzpicture}
\) is a Maurer-Cartan element. The composition in \(\BVGra\) now also defines a (right, by convention) action
\[
  \mathsf{Tw}\,\BVGra\otimes \Def\bigl(\Sigma\Lie \to \BVGra\bigr) \to 
  \mathsf{Tw}\,\BVGra, \;\Gamma\otimes \gamma \mapsto \Gamma \mathbin{\Cdot[2]}\gamma.
\]
of the deformation complex by operadic derivations, by summing over all ways to compose \(\gamma\) into some black vertex of \(\Gamma\).

For any operad \(O\), the elements \(\gamma\in O(1)\) in arity one act on the operad by operadic derivations via the adjoint action (on a \(\Gamma\in O(n)\), say)
\[
  [\gamma, \Gamma] = \gamma\circ_1\Gamma - (-1)^{\vert\gamma\vert \vert\Gamma\vert} \sum_{i=1}^n \Gamma\circ_i\gamma.
\]
In this way we get a differential
\[
  \partial = \bigl[\,
  \begin{tikzpicture}[baseline=0.35ex,shorten >=0pt,auto,node distance=1cm]
    \node[ve] (w1) {};
    \node[vi] (b1) [above of=w1] {};
    \path[every node/.style={font=\sffamily\small}]
      (b1) edge (w1);
  \end{tikzpicture}
  \,, \;\;\bigr]
  + (\;\;)\mathbin{\Cdot[2]}
  \begin{tikzpicture}[baseline=0.35ex,shorten >=0pt,auto,node distance=1cm]
    \node[vi] (w1) {};
    \node[vi] (b1) [above of=w1] {};
    \path[every node/.style={font=\sffamily\small}]
      (b1) edge (w1);
  \end{tikzpicture}
\]
on \(\mathsf{Tw}\,\BVGra\). The construction of the operad \(\mathsf{Tw}\,\BVGra\) is a special case of the operadic ``twisting'' developed in \cite{Willwacher15}.

Define the \emph{operad of BV graphs} as follows. Consider the subspace of \(\mathsf{Tw}\,\BVGra(n)\) spanned by graphs whose black vertices are at least trivalent. Let \(\BVGraphs(n)\) be the quotient of this subspace by the subspace of graphs with a tadpole at a black vertex. Thus, we can picture elements of the operad of BV graphs as series of graphs with black and white vertices and no tadpoles at a black vertex. The reason we define it as a quotient instead is that it makes the formula for the differential easier to describe succinctly, while still behaving as it should. Namely, note that
\[
  \partial\!\!
  \begin{tikzpicture}[baseline=0ex,shorten >=0pt,auto,node distance=1cm]
    \node[ve] (w) {};
    \path[every node/.style={font=\sffamily\small}]
      (w) edge[out=135, in=45, loop] (w); 
  \end{tikzpicture}
  \!\!=\!\!
  \begin{tikzpicture}[baseline=0.5ex,shorten >=0pt,auto,node distance=1cm]
    \node[ve] (w) {};
    \node[vi] (b) [above of=w1] {};
    \path[every node/.style={font=\sffamily\small}]
      (b) edge (w)
      (b) edge[out=135, in=45, loop] (b); 
  \end{tikzpicture}\!\!\!\in \mathsf{Tw}\,\BVGra,
\]
but to make the following lemma true we need the simple tadpole graph to be a cocycle.

\begin{lem}
The association
\begin{align*}
  \mathsf{BV} &\to \BVGraphs,\\
    \wedge,\ [\,,\,],\ \Delta &\mapsto 
    \begin{tikzpicture}[baseline=-0.35ex,shorten >=0pt,auto,node distance=1cm]
      \node[ve] (w1) {};
      \node[ve] (w2) [right of=w1] {};
    \end{tikzpicture},\
    \begin{tikzpicture}[baseline=-0.35ex,shorten >=0pt,auto,node distance=1cm]
      \node[ve] (w1) {};
      \node[ve] (w2) [right of=w1] {};
      \path[every node/.style={font=\sffamily\small}]
        (w1) edge[out=55, in=125] (w2);
    \end{tikzpicture},\!
    \begin{tikzpicture}[baseline=-0.15ex,shorten >=0pt,auto,node distance=1cm]
      \draw [use as bounding box, transparent, white] (-2mm,-2mm) rectangle (2mm,2mm);
      \node[ve] (w) {};
      \path[every node/.style={font=\sffamily\small}]
        (w) edge[out=135, in=45, loop] (w); 
    \end{tikzpicture}\!,
\end{align*}
is a quasi-isomorphism of operads.
\end{lem}

\begin{proof}
Let \(\Graphs\) be the suboperad spanned by graphs containing no tadpoles. That the association \(
    \wedge,\ [\,,\,] \mapsto 
    \begin{tikzpicture}[baseline=-0.35ex,shorten >=0pt,auto,node distance=1cm]
      \node[ve] (w1) {};
      \node[ve] (w2) [right of=w1] {};
    \end{tikzpicture},\
    \begin{tikzpicture}[baseline=-0.35ex,shorten >=0pt,auto,node distance=1cm]
      \node[ve] (w1) {};
      \node[ve] (w2) [right of=w1] {};
      \path[every node/.style={font=\sffamily\small}]
        (w1) edge[out=55, in=125] (w2);
    \end{tikzpicture}\,
\) is a quasi-isomorphism \(\Ger\to\Graphs\) from the Gerstenhaber operad was proved by Kontsevich in \cite{Kontsevich99}. Identify \(\BVGraphs(n)\) with the complex \(\Graphs(n)\otimes \Q[s_1,\dots, s_n]\) (each tadpole \(s_k\) of degree one) and use the well-known relationship between the Gerstenhaber and Batalin-Vilkovisky operads.
\end{proof}

\subsection{Hochschild cohomology of the BV operad}

We shall compute the degree one cohomology of the deformation complex of the obvious morphism \(\Ass\to \mathsf{BV}\). To do so we replace the operad of Batalin-Vilkovisky algebras by the quasi-isomorphic operad of Batalin-Vilkovisky graphs. Note first that there is a splitting of \(\Com\)-modules
\[
  \BVGraphs = \Com \oplus \mathscr{I}\BVGraphs,
\]
with \(\Com\) here denoting the collection of graphs without edges. The contribution to the deformation complex cohomology of these one-dimensional summands is easy to deduce. For the other summand, we note that as a collection we can write it as a left free module on a right cofree module,
\[
  \mathscr{I}\BVGraphs = \Com\infcirc X^{\Com},
\]
cf.~\ref{sec:rightmod}. First, every graph with \(n\) white vertices is a graph with some maximal set \(I\subset [n]\) of white vertices that all have valency at least one, so the factor \(X^{\Com}\) in
\[
  \Com\infcirc X^{\Com}(n) = \bigoplus_{I\subset [n]} \Com(n/I)\otimes X^{\Com}(I)
\]
can be understood as the collection of graphs with all white vertices at least univalent. Secondly, every such graph, with say \(k\) white vertices, can be obtained from a graph with no white vertices at all but with some number \(m\geq k\) of labelled ``hairs'' (missing endpoints of edges) together with a partition of \([m]\) into \(k\) parts that defines how to connect the \(m\) hairs to \(k\) white vertices. The domains \(\Com^{\sconv k}(m)\) in
\[
  X^{\Com}(k) = \prod_{m\geq k}\Map_{\bS_m}\!\bigl(\Com^{\sconv k}(m), X(m)\bigr)
\]
are exactly such partitions, so we can take \(X(m)\) to be the space of graphs with only black vertices (of valency at least three) and \(m\) labelled hairs.

Note that we must include the graph
\[\!
  \begin{tikzpicture}[baseline=0.8ex,shorten >=0pt,auto,node distance=1cm]
    \node (w1) {};
    \node (w2) [right of=w1] {};
    \path[every node/.style={font=\sffamily\small}]
      (w1) edge[out=55, in=125] (w2);
  \end{tikzpicture}\!\!\in X(2),
\]
(which has two hairs and one edge) and all graphs containing it as a connected component. Together with the partition \([2]=\{1\}+\{2\}\) this graph without vertices produces \(
  \begin{tikzpicture}[baseline=-0.35ex,shorten >=0pt,auto,node distance=1cm]
    \node[ve] (w1) {};
    \node[ve] (w2) [right of=w1] {};
    \path[every node/.style={font=\sffamily\small}]
      (w1) edge[out=55, in=125] (w2);
  \end{tikzpicture}
\) while coupled with the partition \([2]=\{1,2\}\) it yields the tadpole \(
  \begin{tikzpicture}[baseline=-0ex,shorten >=0pt,auto,node distance=.4cm]
    \draw[use as bounding box, transparent, white] (-2mm,-2mm) rectangle (2mm,2mm);
    \node[ve] (w) {};
    \path[every node/.style={font=\sffamily\small}]
      (w) edge[out=135, in=45, loop] (w); 
  \end{tikzpicture}
\).

Arguing as in \ref{sec:harrcofree} gives an identification as graded vector spaces
\begin{align*}
  \Def(\Ass\to &\mathscr{I}\BVGraphs) \cong\\ &\prod_{m\geq 1}
  \susp{}{\textstyle\mathsf{gr}^{(1,\dots,1)}}\,\Ass\Bbrackets{\susp{-1}\Q[x_1,\dots, x_m]}
  \otimes_{\bS_m}\! X(m).
\end{align*}
The only difference, due to taking Hoschschild instead of Harrison cohomology, is that we consider the free associative algebra \(\Ass\bbrackets{\dots}\) instead of the free Lie algebra \(\Lie\bbrackets{\dots}\), plus the obvious change in grading and inclusion of arity one that is due to our conventions on deformation cohomology versus bimodule cohomology.

The identification \(\mathscr{I}\BVGraphs = \Com\infcirc X^{\Com}\) is almost true as infinitesimal \(\Com\)-bimodules. It holds as infinitesimal left modules. The action \((-)\circ_i m_2\) on the cofree right module \(X^{\Com}\) is given by summing over all ways to partition the \(i\)th subset into two nonempty subsets, which interpreted in \(\BVGraphs\) amounts to splitting the \(i\)th white vertex into two and summing over all ways to reconnect edges \emph{such that neither of the two white vertices is zero-valent}. However, the actual right action on \(\BVGraphs\) does allow all edges to reconnect to just one, creating a zero-valent white vertex. The conclusion is that just the Hochschild part of the differential acts exactly like the non-reduced cobar differential. By the Hochschild-Kostant-Rosenberg theorem the cohomology of the cobar construction is
\[
  H^*\bigl(\Ass\bbrackets{\susp{-1}\Q[x_1,\dots, x_m]}\bigr) = 
  \Q\bbrackets{\susp{-1}x_1,\dots, \susp{-1}x_m}.
\]
Hence, the cohomology of the associated graded for the filtration by the internal grading on Batalin-Vilkovisky graphs is 
\[
  \prod_{n\geq 1} s\hspace{.8pt}\Q\,\susp{-1}x_1\wedge\dots\wedge\susp{-1}x_n \otimes_{\bS_n}\! X(n)
\]
and hence the cohomology of the deformation complex is concentrated in the subcomplex \(\mathit{fC}\) spanned by graphs that have all white vertices of valency exactly one and which are (anti)symmetric under exchange of white vertices. In particular, \emph{this subcomplex contains no graphs with tadpoles}.

\begin{prop}\label{prop:HochBV}
\(H^1(\Def(\As_{\infty}\to\mathit{BV})) =
  \Q\,
  \begin{tikzpicture}[baseline=-.1ex,shorten >=0pt,auto,node distance=.7cm]
    \node[ve] (w1) {};
    \node[ve] (w2) [right of=w1] {};
    \node[ve] (w3) [right of=w2] {};
    \node[vi] (z) [above of=w2] {};  
    \path[every node/.style={font=\sffamily\small}]
      (w1) edge (z)
      (w2) edge (z)
      (w3) edge (z); 
  \end{tikzpicture}
  \,\oplus \mathfrak{grt}_1\).
\end{prop}

\begin{proof}
Turchin and Willwacher prove (in \cite{TurchinWillwacher18}) the same statement but with \(\mathit{BV}\) replaced by the operad \(\Ger\) of Gerstenhaber algebras. Their proof proceeds as follows. First of all they replace the operad of Gerstenhaber algebras by Kontsevich's operad \(\Graphs\), which can be thought of as the suboperad of \(\BVGraphs\) that is spanned by graphs without tadpoles. They then follow arguments similar to those outlined above to reduce to a subcomplex \(\mathit{fC}\)--which is isomorphic to our complex \(\mathit{fC}\), above.
\end{proof}

\noindent
In terms of bimodule cohomology, \(
H^0(\Ass,\mathsf{BV}) =
  \Q\,
    \begin{tikzpicture}[baseline=-0.35ex,shorten >=0pt,auto,node distance=1cm]
      \node[ve] (w1) {};
      \node[ve] (w2) [right of=w1] {};
    \end{tikzpicture}
  \,\oplus\Q\,
  \begin{tikzpicture}[baseline=-.1ex,shorten >=0pt,auto,node distance=.7cm]
    \node[ve] (w1) {};
    \node[ve] (w2) [right of=w1] {};
    \node[ve] (w3) [right of=w2] {};
    \node[vi] (z) [above of=w2] {};  
    \path[every node/.style={font=\sffamily\small}]
      (w1) edge (z)
      (w2) edge (z)
      (w3) edge (z); 
  \end{tikzpicture}
  \,\oplus \mathfrak{grt}_1
\). That is to say, removing arity one from the deformation complex introduces the class \(
  \begin{tikzpicture}[baseline=-0.35ex,shorten >=0pt,auto,node distance=1cm]
    \node[ve] (w1) {};
    \node[ve] (w2) [right of=w1] {};
  \end{tikzpicture}
\) and changes the representative of every graph representing a Grothendieck-Teichm\"uller element, but can not make them exact.

\subsubsection{The Hodge-type grading}

Recall from \ref{sec:HarrisonHodge} the Hodge grading on the deformation complex of the morphism \(\Ass\to\BVGraphs\). Since the Eulerian idempotent \(e^{\hspace{.5pt}n}_n\) is projection onto the totally antisymmetric part, the Hodge degree of a graph in the complex
\[
  \mathit{fC} \cong \prod_{n\geq 1} s\hspace{.8pt}\Q\,\susp{-1}x_1\wedge\dots\wedge\susp{-1}x_n \otimes_{\bS_n} \!X(n),
\]
discussed at the end of the preceding subsection, is equal to the number \(n\) of white vertices. The classes corresponding to Grothendieck-Teichm\"uller elements are all represented by graphs with a single white univalent vertex (as argued in \cite{Willwacher15} and \cite{TurchinWillwacher18}), so these are in Hodge degree one. Indeed, we have already interpreted the Grothendieck-Teichm\"uller Lie algebra as a Harrison cohomology. The only other Hochshchild cohomology class, represented by \(
  \begin{tikzpicture}[baseline=-.1ex,shorten >=0pt,auto,node distance=.7cm]
    \node[ve] (w1) {};
    \node[ve] (w2) [right of=w1] {};
    \node[ve] (w3) [right of=w2] {};
    \node[vi] (z) [above of=w2] {};  
    \path[every node/.style={font=\sffamily\small}]
      (w1) edge (z)
      (w2) edge (z)
      (w3) edge (z); 
  \end{tikzpicture}
\), has Hodge degree three.

\section{Proof of the main theorem}

This section contains the proof of our main result:

\begin{thm*}
The regularised Knizhnik-Zamolodchikov connection induces an isomorphism between \(\Q\triang\oplus\Q\pent\oplus\mathfrak{grt}_1\) and the degree zero Hochschild cohomology of the homology operad of Brown's dihedral moduli spaces.
\end{thm*}

\subsection{Lie algebra morphisms}

Define (w.l.o.g, \(1\leq i < j\leq n\))
\[
  \pi: \mathfrak{rb}(n) \to \mathfrak{d}(n),\;\left\{
  \begin{aligned}
    \, b_{ij} &\mapsto \delta_{i\,j} + \delta_{i+1\,j-1} - \delta_{i+1\,j} - \delta_{i\,j-1}, \\
    \, s_k\, &\mapsto 0.
  \end{aligned}\right.
\]
This is the composite of the projection \(\mathfrak{rb}(n) \to \mathfrak{p}(n)\) that sends all the \(s_k\) to zero and the isomorphism \(\mathfrak{p}(n)\cong \mathfrak{d}(n)\), cf.~\ref{sec:braids}. Define also
\[
  \gamma: \mathfrak{d}(n) \to \mathfrak{rb}(n),\;
  \delta_{ij} \mapsto \!\sum_{i\leq r<s\leq j} b_{rs} + \sum_{i\leq k\leq j} s_k.
\]
By rephrasing the proof of isomorphism \(\mathfrak{p}(n)\cong \mathfrak{d}(n)\) one sees \(\pi\circ\gamma = \id\). Both these functions are morphisms of collections of Lie algebras.

\subsubsection{Morphism of bimodules}

The morphism \(\pi\) is not  morphism of operads of Lie algebras. However, it does induce a morphism of infinitesimal \(\Com\)-bimodules
\[
  \pi^{\delta} : H_*(\mathfrak{rb}) \xrightarrow{\pi} H_*(M_0) \to H_*(M^{\delta}).
\]
We are here identifying the cohomology of the open moduli space with the Lie algebra cohomology of the dihedral braid Lie algebra, using the Knizhnik-Zamolodchikov connection. We verify this explicitly. Let \(\partial_i = (-)\circ_i m_2\). Then
\[
\begin{aligned}
  &(\pi\circ \partial_i-\partial_i\circ\pi)(s_i) = \delta_{i\,i+1} - 0, \\
  &(\pi\circ \partial_i - \partial_i\circ\pi)(b_{i-1\,i}) = (\delta_{i-1\,i+1} - \delta_{i\,i+1}) - \delta_{i-1\,i+1},\\  
  &(\pi\circ \partial_i - \partial_i\circ\pi)(b_{i\,i+1}) = (\delta_{i\,i+2} - \delta_{i\,i+1}) - \delta_{i\,i+2}.
\end{aligned}
\]
The discrepancy is in all cases \(\delta_{i\,i+1}\). In all cases except those listed we have \(\pi\circ\partial_i = \partial_i\circ\pi\). Interpreted in terms of chord diagrams this means that for all \(B\in H_*(\mathfrak{rb}(n-1))\) the difference \((\pi\circ \partial_i - \partial_i\circ \pi)(B)\in H_*(\mathfrak{d}(n))\) is a sum of chord diagrams of the form \(c_{i\,i+1}\wedge A\) where \(A\) is a diagram on the \(n\)-gon formed by cutting \(c_{i\,i+1}\) on the \((n+1)\)-gon. In particular, the difference consists of diagrams that are not prime.

The projection \(\pi\) commutes on the nose with the two left actions of the binary generator; \(\partial_0 = m_2\circ_2(-)\) and \(\partial_{n+1} = m_2\circ_1(-)\).

\subsubsection{Morphism of operads}\label{sec:gamma} 

The morphism \(\gamma\) is a morphism of operads. This was proved by the author in \cite{Alm16}. The proof is simple once it is noticed that
\(
  \gamma(\delta_c) = s_I \circ_I 0\) for \(
  \circ_I: \mathfrak{rb}(n/I) \oplus \mathfrak{rb}(I) \to \mathfrak{rb}(n),
\)
writing \(I\) for both the interval enclosed by the chord \(c\) and the corresponding element of \(n/I\). 

\subsection{Proof of the main theorem}\label{sec:theproof}

Consider the following diagram:
\[
\begin{tikzcd}
 & H^0(\As,\mathsf{BV})\ar{r}{\pi^{\delta}} & H^0(\As, H_*(M^{\delta})) \ar{dd}{\reg^{\vee}} \\
\Q^{\oplus 2}\oplus\mathfrak{grt}_1\ar{ur}{\cong}\ar{dr}{} & \\
 & H^0(\As, H_*(M_0))\ar{uu}{\gamma} & H^0(\As,\gr{}H_*(M_0))\mathrlap{.} \ar{l}{\cong}
\end{tikzcd}
\]
We know by proposition \ref{prop:HochBV} that
\(H^0(\As,\mathsf{BV}) =
  \Q\,
    \begin{tikzpicture}[baseline=-0.35ex,shorten >=0pt,auto,node distance=1cm]
      \node[ve] (w1) {};
      \node[ve] (w2) [right of=w1] {};
    \end{tikzpicture}
  \,\oplus\Q\,
  \begin{tikzpicture}[baseline=-.1ex,shorten >=0pt,auto,node distance=.7cm]
    \node[ve] (w1) {};
    \node[ve] (w2) [right of=w1] {};
    \node[ve] (w3) [right of=w2] {};
    \node[vi] (z) [above of=w2] {};  
    \path[every node/.style={font=\sffamily\small}]
      (w1) edge (z)
      (w2) edge (z)
      (w3) edge (z); 
  \end{tikzpicture}
  \,\oplus \mathfrak{grt}_1\), i.e., that the top map in the triangle on the left is an isomorphism. The triangle commutes because given \(\psi(x,y)\in\mathfrak{grt}_1\),
\[
  \gamma(\psi(\delta_{12},\delta_{23})) = \psi(b_{12}+s_1+s_2, b_{23} + s_2 + s_3) = \psi(b_{12}, b_{23}),
\]
as the \(s_k\)'s are central and \(\psi\) has nothing in depth zero. Thus \(\gamma\) is onto and \(\Q^{\oplus 2} \oplus\mathfrak{grt}_1\) injects into \(H^0(\As, H_*(M_0))\).

The composite of \(\reg^{\vee}\) and the bottom horizontal map is an isomorphism by theorem \ref{thm:isohoch}. Since regularisation is a splitting on the associated graded cooperad (cf.~\ref{sec:regcoop}) and \(\pi\circ\gamma=\id\), the effect of going once around the diagram beginning at the top right corner is the identity.

It follows that
\(
  H^0(\As,H_*(M^{\delta})) \to H^0(\As, \mathsf{BV})
\)
is injective, so \(\gamma\) is injective. Thus both \(\gamma\) and \(\pi^{\delta}\) are isomorphisms.  We conclude that the map from
\(
  \Q\triang \oplus\Q\pent\oplus\mathfrak{grt}_1\) to \(H^0(\As,H_*(M^{\delta})),
\)
from  \ref{thm:grtinj} is an isomorphism.

\subsubsection{Zeta of two, again}

In the complex \(C(\As,\BVGraphs)\),
\begin{align*}
  \partial\hspace{1pt} \biggl(\,
  \begin{tikzpicture}[baseline=.2ex,shorten >=0pt,auto,node distance=1.5cm]
    \node[ve] (w1) {};
    \node[ve] (w2) [right of=w1] {};
    \node[ve] (w3) [right of=w2] {};  
    \path[every node/.style={font=\sffamily\small}]
      (w1) edge[out=35, in=145] (w2)
      (w1) edge[out=75, in=105] (w3); 
  \end{tikzpicture}
  \,+\,
  \begin{tikzpicture}[baseline=.2ex,shorten >=0pt,auto,node distance=1.5cm]
    \node[ve] (w1) {};
    \node[ve] (w2) [right of=w1] {};
    \node[ve] (w3) [right of=w2] {}; 
    \path[every node/.style={font=\sffamily\small}]
      (w1) edge[out=75, in=105] (w3)
      (w2) edge[out=35, in=145] (w3); 
  \end{tikzpicture} 
  \,\biggr)
  &= 2\;
  \begin{tikzpicture}[baseline=.5ex,shorten >=0pt,auto,node distance=1.5cm]
    \node[ve] (w1) {};
    \node[ve] (w2) [right of=w1] {};
    \node[ve] (w3) [right of=w2] {};
    \node[vi] (z) [above of=w2] {};  
    \path[every node/.style={font=\sffamily\small}]
      (w1) edge (z)
      (w2) edge (z)
      (w3) edge (z); 
  \end{tikzpicture}\,,\\
  \partial_H \biggl(\,
  \begin{tikzpicture}[baseline=.2ex,shorten >=0pt,auto,node distance=1.5cm]
    \node[ve] (w1) {};
    \node[ve] (w2) [right of=w1] {};
    \node[ve] (w3) [right of=w2] {};  
    \path[every node/.style={font=\sffamily\small}]
      (w1) edge[out=35, in=145] (w2)
      (w1) edge[out=75, in=105] (w3); 
  \end{tikzpicture}
  \,+\,
  \begin{tikzpicture}[baseline=.2ex,shorten >=0pt,auto,node distance=1.5cm]
    \node[ve] (w1) {};
    \node[ve] (w2) [right of=w1] {};
    \node[ve] (w3) [right of=w2] {};
    \path[every node/.style={font=\sffamily\small}]
      (w1) edge[out=75, in=105] (w3)
      (w2) edge[out=35, in=145] (w3); 
  \end{tikzpicture} 
  \,\biggr)
  &= 2\;
  \begin{tikzpicture}[baseline=.2ex,shorten >=0pt,auto,node distance=1.5cm]
    \node[ve] (w1) {};
    \node[ve] (w2) [right of=w1] {};
    \node[ve] (w3) [right of=w2] {};
    \node[ve] (w4) [right of=w3] {};
    \path[every node/.style={font=\sffamily\small}]
      (w1) edge[out=75, in=105] (w3)
      (w2) edge[out=75, in=105] (w4); 
  \end{tikzpicture}\,.
\end{align*}
Here \(\partial_H\) is the Hochschild differential. It follows that the graphs \(
  \begin{tikzpicture}[baseline=-.1ex,shorten >=0pt,auto,node distance=.7cm]
    \node[ve] (w1) {};
    \node[ve] (w2) [right of=w1] {};
    \node[ve] (w3) [right of=w2] {};
    \node[vi] (z) [above of=w2] {};  
    \path[every node/.style={font=\sffamily\small}]
      (w1) edge (z)
      (w2) edge (z)
      (w3) edge (z); 
  \end{tikzpicture}
\) and \(
  \begin{tikzpicture}[baseline=-.2ex,shorten >=0pt,auto,node distance=.7cm]
    \node[ve] (w1) {};
    \node[ve] (w2) [right of=w1] {};
    \node[ve] (w3) [right of=w2] {};
    \node[ve] (w4) [right of=w3] {};
    \path[every node/.style={font=\sffamily\small}]
      (w1) edge[out=75, in=105] (w3)
      (w2) edge[out=75, in=105] (w4); 
  \end{tikzpicture}
\) are cohomologous. Equivalence of the latter with \(\pent\) is easily deduced.

\section{Context and consequences}

This section explains some salient consequences of our main theorem, and places the result in the wider context of our research field.

\subsection{Indecomposable multiple zeta values}\label{sec:QZ}

We here outline how our results elucidate the (partly conjectural) relationship between the Grothendieck-Teichm\"uller Lie algebra and the indecomposable quotient of the algebra of multiple zeta values. More could be said; we only outline the most important aspects. The questions touched upon are problems we hope to develop further in future work.

\subsubsection{Embedded associahedra}\label{sec:associahedra}

Let \(X(n)\subset M^{\delta}(\R)(n)\) be the cell defined by the inequalities \(0\leq u_c\leq 1\), for all chords \(c\in\chi_1(n)\). It is diffeomorphic as a semialgebraic manifold with corners to the \((n-2)\)-dimensional associaheder. Furthermore, the operad structure on the dihedral moduli spaces restricts to define an operad structure on the collection of these cells, such that the associated differential graded operad of fundamental chains is isomorphic to the nonsymmetric Koszul resolution \(\As_{\infty}\). 

Integration pairing thus defines a morphism
\[
  \Phi:\R\otimes\As_{\infty} \to \R\otimes H_*(M^{\delta}), \; m_n \mapsto \sum_{P\in P(n)} \int_{X(n)} \alpha_P \otimes P.
\]
In the above formula \(P(n)\) is the set of top degree prime gravity chord diagrams. By Brown's theorem \ref{thm:brownperiods}, all of the involved integrals are rational linear combinations of multiple zeta values, meaning it is enough to extend coefficients from \(\Q\) to the algebra \(Z\) of multiple zeta values. Moreover, every multiple zeta value occurs as such an integral.

Note that \(m_2\) maps to \(\triang\), \(m_3\) to zero, \(m_4\) to \(\zeta(2)\pent\),
\[
  m_5 \mapsto \zeta(3)\biggl(
  \begin{tikzpicture}[font=\scriptsize, baseline={([yshift=-.6ex]current bounding box.center)}]
    \node[regular polygon, shape border rotate=60,
          regular polygon sides=6, minimum size=.75cm, draw] at (0,0) (A) {};
    \path[thick] (A.corner 1) edge (A.corner 4)
                 (A.corner 1) edge (A.corner 5)
                 (A.corner 2) edge (A.corner 6);
  \end{tikzpicture}
  \,+\,
  \begin{tikzpicture}[font=\scriptsize, baseline={([yshift=-.6ex]current bounding box.center)}]
    \node[regular polygon, shape border rotate=60,
          regular polygon sides=6, minimum size=.75cm, draw] at (0,0) (A) {};
    \path[thick] (A.corner 1) edge (A.corner 4)
                 (A.corner 2) edge (A.corner 5)
                 (A.corner 2) edge (A.corner 6);
  \end{tikzpicture}
  \,+\,
  \begin{tikzpicture}[font=\scriptsize, baseline={([yshift=-.6ex]current bounding box.center)}]
    \node[regular polygon, shape border rotate=60,
          regular polygon sides=6, minimum size=.75cm, draw] at (0,0) (A) {};
    \path[thick] (A.corner 1) edge (A.corner 5)
                 (A.corner 2) edge (A.corner 5)
                 (A.corner 3) edge (A.corner 6);
  \end{tikzpicture}
  \,+\,
    \begin{tikzpicture}[font=\scriptsize, baseline={([yshift=-.6ex]current bounding box.center)}]
    \node[regular polygon, shape border rotate=60,
          regular polygon sides=6, minimum size=.75cm, draw] at (0,0) (A) {};
    \path[thick] (A.corner 1) edge (A.corner 5)
                 (A.corner 2) edge (A.corner 6)
                 (A.corner 3) edge (A.corner 6);
  \end{tikzpicture}
  \biggr),
\]
and so on. The only relations between the multiple zeta values implied by the morphism \(\Phi\) are the quadratic relations due to Stokes' theorem:
\[
  0 = \int_{\partial X(n)} \alpha_P = \sum_{c\in\chi_1(n)} \int_{X(n/I)\times X(I)} \Reg_c(\alpha_P),
\]
for \(P\) a prime gravity chord diagram of top degree minus one.

\subsubsection{Motivic version}

The Stokes' relations mentioned above are motivic, so it is fairly obvious that \(\Phi\) can be lifted to a morphism that has motivic periods as coefficients. We can actually be very explicit about this. Let \(D(n) = M^{\delta}(n)\setminus M_0(n)\) be the boundary normal crossing divisor and note that the associaheder \(X(n)\) has boundary on \(D(n)\). Define, for every prime gravity chord diagram \(P\) of top degree on the \((n+1)\)-gon, the framed (mixed Tate, unramified over \(\Z\)) motive
\[
  I_P = \bigl[X(n), \alpha_P, H^{n-2}\bigl(M^{\delta}(n), D(n)\bigr) \bigr].
\]
The period of this motive is the integral of \(\alpha_P\) over \(X(n)\). The Stokes' relations can be lifted to motivic relations that can be written entirely in terms of these framed motives. Let \(Z^{\mathfrak{m}}\) be the subalgebra of the algebra of all motivic periods of mixed Tate motives (unramfied over the integers), generated by the \(I_P\)'s. We then have a morphism
\[
  \Phi^{\mathfrak{m}} : Z^{\mathfrak{m}}\otimes \As_{\infty} \to Z^{\mathfrak{m}}\otimes H_*(M^{\delta}),\;
  m_n \mapsto \sum_{P\in P(n)} I_P \otimes P.
\]

\subsubsection{Maurer-Cartan interpretation}\label{sec:MC}

Set \(\Phi^{\mathfrak{m}}_+ = \Phi^{\mathfrak{m}} - \triang\). Thus, \(\Phi^{\mathfrak{m}}_+\) is a Maurer-Cartan element in the deformation complex
\[
  Z^{\mathfrak{m}} \otimes \Def\bigl(\As_{\infty} \xrightarrow{\scalebox{.7}{\triang}} H^*(M^{\delta})\bigr).
\]
Equivalently, it is a morphism \(C^*(\Def(\triang)) \to Z^{\mathfrak{m}}\) of differential graded algebras. First note that it takes values in the augmentation ideal \(Z^{\mathfrak{m}}_+\) of motivic multiple zeta values of strictly positive weights. The bracket \([\Phi^{\mathfrak{m}}_+,\Phi^{\mathfrak{m}}_+]\) thus has coefficients in the ideal \(Z^{\mathfrak{m}}_+ \cdot Z^{\mathfrak{m}}_+\) of nontrivial products, so modulo this ideal we obtain a morphism
\[
  \phi^{\mathfrak{m}} : s^{-1}\Def(\triang)^{\!\vee} \to \mathfrak{nz} = Z^{\mathfrak{m}}_+  / (Z^{\mathfrak{m}}_+)^2
\]
of differential graded vector spaces. The complex \(s^{-1}\Def(\triang)^{\vee}\) can be identified with the dual Hochschild cochain complex of \ref{sec:dualhoch}, i.e., it is spanned by prime gravity chord diagrams and is equipped with the corner-cutting differential. The morphism \(\phi^{\mathfrak{m}}\) is surjective, by construction. Let \(K\) denote its kernel. Then \(K^1=0\) for degree reasons so the corresponding morphism 
\[
  \varphi^{\mathfrak{m}} : H^0(s^{-1}\Def\bigl(\triang)^{\!\vee}\bigr) \to \mathfrak{nz}
\]
on cohomology is also surjective. It was the simple conceptual reasoning up to this conclusion that motivated the present paper. By our main theorem, we conclude that we have a surjection
\[
  \varphi^{\mathfrak{m}}: \Q\pent \oplus \mathfrak{grt}_1^{\vee} \to \mathfrak{nz}.
\]
The Drinfeld conjecture says that \(\mathfrak{grt}_1\) is isomorphic to a free completed Lie algebra on one generator in each odd degree. Since \cite{Brown12} proves that the indecomposables of the motivic multiple zeta values is dual to such a Lie algebra plus a one-dimensional piece corresponding to zeta of two, the conjecture is (a posteriori) equivalent to the conjecture that \(\varphi^{\mathfrak{m}}\) is an isomorphism.

\subsubsection{Conjectural isomorphism}\label{sec:conjecture}

That the map \(\varphi^{\mathfrak{m}}\) should be an isomorphism translates to \(H^0(K)=0\), something we can reformulate informally as the slogan: If a sum of prime chord diagrams is a shuffle product, then it ought to be exact for the corner-cutting differential. This informal statement can be formulated very concretely, using Brown's break-through result in \cite{Brown12}, that \(\mathfrak{nz}\) has a basis consisting of multiple zeta values corresponding to Lyndon words in the alphabet \(\{2,3\}\). Every such Lyndon word \(w\) defines a prime gravity chord diagram \(P_w\), so the Drinfeld conjecture can be reformulated combinatorially as follows:

\begin{conj*}
The prime gravity chord diagrams corresponding to Lyndon words in the alphabet \(\{2,3\}\) give a basis for the space of top degree prime gravity chord diagrams modulo the image of the corner-cutting differential.
\end{conj*}

\subsection{Relative non-formality}
The collection \(X = \{ X(n) \}\) of the embedded associahedra constitute a copy of Stasheff's topological \(A_{\infty}\) operad. We note that the embedding \(\iota: X \to M^{\delta}(\C)\) of operads induces the map \(m_2\mapsto\triang\) on the associated homology operads. Now, the operad \(X\) is formal. It is proved in \cite{AlmPetersen17,DupontVallette17} that the operad of Brown's moduli spaces is formal, too. Recall that these formality-statements mean that there exists differential graded operads \(P\) and \(Q\), and zig-zags of quasi-isomorphisms
\[
  C_*(X) \leftarrow P \to H_*(X)\;\;\text{and}\;\; C_*(M^\delta) \leftarrow Q \to H_*(M^\delta).
\]
It is natural to ask if the embedding \(\iota: X \to M^{\delta}(\C)\) is formal as a morphism of operads, in the sense that the formalities can be designed in such a way as to fit in a commutative diagram
\[
\begin{tikzcd}
C_*(X)\ar{r}{\iota_*} & C_*(M^{\delta}) \\
P\ar{u}{\simeq}\ar{r}\ar{d}[swap]{\simeq} & Q\ar{u}[swap]{\simeq}\ar{d}{\simeq} \\
H_*(X)\ar{r}{\scalebox{.7}{\triang}} & H_*(M^{\delta})\mathrlap{.}
\end{tikzcd}
\]
Let us call this \emph{relative formality}.

Our results prove that \(\iota:X\to M^{\delta}(\C)\) is \emph{not} relatively formal. Namely, \(\iota\) is on the chain-level equivalent to the morphism \(\Phi^{\mathfrak{m}}\) of \ref{sec:MC}, which is equal to the non-trivial deformation \(\Phi^{\mathfrak{m}}_+\) of \(\triang\). To see that \(\iota\) indeed corresponds to \(\Phi^{\mathfrak{m}}\) one can argue as follows. Take chains to mean currents. Then there is a quasi-isomorphism \(C_*(M^{\delta})\to \R\otimes H_*(M^{\delta})\), identifying the latter as real-valued functionals on the explicit algebra of differential forms \(H^*(M^{\delta})\). Mapping a fundamental cell \([X(n)]\) to the current given by integration over it defines a quasi-isomorphism \(\R\otimes \As_{\infty}\to C_*(X)\). The composite
\[
  \R\otimes \As_{\infty}\to C_*(X) \xrightarrow{\iota_*} C_*(M^{\delta})\to \R\otimes H_*(M^{\delta})
\]
exactly the morphism \(\Phi\) given in \ref{sec:associahedra}, which in turn is \(\Phi^{\mathfrak{m}}\) under period evaluation. Note that the role of integration pairing played in the argument of relating \(\iota\) to some map \(\mathbb{k}\otimes\As_{\infty} \to \mathbb{k}\otimes H_*(M^{\delta})\), for a \(\Q\)-algebra \(\mathbb{k}\), could be replaced by some other comparison isomorphism (now over \(\mathbb{k}\)) between Betti and de Rham cohomology, which would correspond to a \(\mathbb{k}\)-valued evaluation of the coefficients of \(\Phi^{\mathfrak{m}}\).

\section{The Tamarkin argument}

The elegant preprint \cite{Tamarkin02} by Tamarkin contains as one of its main results a proof that the degree zero Harrison cohomology of the Gerstenhaber operad equals the Grothendieck-Teichm\"uller Lie algebra, plus a one-dimensional term that corresponds to rescaling the product:
 \[
 H^0(\Com,\Ger) \cong \Q\oplus\mathfrak{grt}_1.
 \]
In this section we reconstruct a crucial step in Tamarkin's proof. We do this because Tamarkin's preprint contains some minor inaccuracies that we take this opportunity to rectify.

\subsection{Right modules}\label{sec:rightmod}

Let \(\Mod{P}\) be the category of right modules for some differential graded operad \(P\) and let \(\Bimod{P}\) be the category of infinitesimal \(P\)-bimodules. We write \(\Coll\) for the category of differential graded collections. In analogy to the classical theory of modules over a ring, there is a functor
\begin{gather*}
  \mathbb{I}\infcirc_{P}- : \Bimod{P} \to \Mod{P},\\
  \mathbb{I}\infcirc_{P}\! M  = 
  \mathsf{Coeq}\bigl(\mathbb{I} \infcirc P \infcirc M \rightrightarrows \mathbb{I}\infcirc M \cong M\bigr),
\end{gather*}
with a right inverse functor 
 \[
  P\infcirc- : \Mod{P} \to \Bimod{P},
 \]
given by adding a free infinitesimal left module action. Moreover, the forgetful functor \(\Mod{P} \to \Coll\) from right modules to collections of differential graded vector spaces has a left adjoint, the free module \((-)\circ P\), and a right adjoint, the cofree (or coinduced) module
\begin{gather*}
  (-)^P : \Coll \to \Mod{P},\\ U^P(n) = \Map_{\bS}(P^{\sconv n}, U).
\end{gather*}
Here \(\Conv\) is the \emph{Day convolution},
\[
  U\Conv V (n) = \int^{I,\hspace{.4pt}J\in\bS} \bS(n, I+J)\otimes U(I)\otimes V({\mkern 2mu}J) = \bigoplus_{[n]=I+J} U(I)\otimes V({\mkern 2mu}J).
\]
The formula for \(U^P\) follows from identifying it as the right Kan extension of \(U\), considered as a presheaf on the groupoid \(\bS\) of finite sets and bijections, to a presheaf on the PROP defined by \(P\), noting that this PROP has as space of maps from \(I\) to \(J\) the object \(P^{\sconv J}(I)\).

\subsubsection{Cofree injective resolutions}

It follows from the adjunctions that if \(U\) is an injective object in the category of differential graded collections, then \(U^P\) is an injective right module. Moreover, if \(N\to U\) is a monomorphism of collections into an injective collection \(U\), then from the natural monic \(N \to N^P\) we deduce a monomorphism \(N\to U^P\) of right modules. Thus, any right module admits an injective resolution by cofree modules. 

\subsection{Cohomology of cofree right modules}

Fix a morphism of differential graded operads \(E \to P\) and assume \(E\) is Koszul. We may regard every infinitesimal \(P\)-bimodule as an infinitesimal \(E\)-bimodule. Then, for \(U\) any differential graded collection,
\begin{align*}
  C(E, P\infcirc U^P) &= \Map_{\bS}\bigl(\susp{-1}E^{\antishriek}_+, P\infcirc \Map_{\bS}(P^{\sconv},U)\bigr) \\
  &= \prod_{m\geq 1}C(E\to P, m)\otimes_{\bS_m}\! U(m),
\end{align*}
where we define
\[
  C(E\to P, m) = \prod_{\substack{n\geq 1 \\ I\subset [n]}} 
  s\Sigma E^{\hspace{.5pt}!}_+(n)\,\otimes_{\bS_n} P(n/I)\otimes \mathit{coP}^{\sconv I}(m).
\]
This decomposition of \(C(E, P\infcirc U^P)\) is clearly functorial and, in particular, it is compatible with differentials.

\subsubsection{The Harrison case}\label{sec:harrcofree}

Abbreviate \(C(\Com, m)=C(\Com\xrightarrow{\id}\Com, m)\). Thus, from the preceding subsection we have
\[
  C(\Com, m) = \prod_{\substack{n\geq 2 \\ I\subset [n]}} 
  \susp{2-n}\bigl(\mathsf{sgn}_n\otimes \Lie(n)\bigr)\otimes_{\bS_n}\!\bigl(\Com(n/I)\otimes \mathsf{coCom}^{\sconv I}(m)\bigr).
\]
The term \(\mathsf{coCom}^{\sconv I}(m)\) is a sum over partitions of \([m]\) into subsets indexed by \(I\). We can represent such by surjections \([m]\to I\). Since we are also summing over inclusions \(I\subset [n]\) and every function is a surjection onto its image, we can combine both sums into a sum over arbitrary functions \(f:[m]\to [n]\). Given such a function, write the formal expression
\[
  \bigl(\prod_{i\in f^{-1}(1)} x_i \mid \dots \mid \prod_{i\in f^{-1}(n)}x_i\bigr).
\]
If \(f^{-1}(k)=\emptyset\) we interpret the product as \(1\). These expressions are a basis for the subspace of \((s^{-1}\Q[x_1,\dots, x_m])^{\otimes n}\) that is of homogeneous degree \(1\) in each variable \(x_i\). Call this subspace \(\mathsf{gr}^{(1,\dots, 1)}\, (s^{-1}\Q[x_1,\dots, x_m])^{\otimes n}\). A function \(f:[m]\to [n]\) can be recovered from the formal tensor expression we associated to it, and together with the choice of grading and natural permutation action, we may identify
\begin{align*}
  C(\Com, m) &= \prod_{n\geq 2} 
    \susp{2} \Lie(n) \otimes_{\bS_n} \mathsf{gr}^{(1,\dots,1)}\bigl(\susp{-1}\Q[x_1,\dots, x_m]\bigr)^{\otimes n} \\
  &= \susp{2}\mathsf{gr}^{(1,\dots,1)}\,\Lie_+\Bbrackets{\susp{-1}\Q[x_1,\dots, x_m]}
\end{align*}
as the subspace of the completed free Lie algebra on the suspension of the cofree cocommutative coalgebra \(\Q[x_1,\dots,x_m]\) that is spanned by formal Lie words that contain at least one bracket, and are of homogeneous degree \((1,\dots,1)\). The full \(\Lie\bbrackets{s^{-1}\Q[x_1,\dots,x_m]}\) is the underlying graded vector space of the completed non-reduced Harrison cobar construction on the cocommutative polynomial coalgebra. However, the differential in the present case is \emph{not} that bar differential.\footnote{\;Tamarkin claims that it is.} The non-reduced cobar differential acts like, 
\[
  (\,a\mid xy \mid b\,) \mapsto \pm\bigl(\,a\mid 1\otimes xy + x\otimes y + y\otimes x + xy\otimes 1\mid b\,\bigr) + \dots 
\]
whereas our differential here will act as
\[
  (\,a\mid x\mid b\,) \mapsto \pm\bigl(\,a\mid x\otimes y + y\otimes x\mid b\,\bigr) + \dots.
\]
Concretely, the part of the differential defined by the right action of \(\Com\) acts on \(x\)'s just like on a reduced Harrison cohomology complex: it cannot create \(1\)'s, \emph{unless it acts on a} 1, in which case it will split it into two \(1\)'s. However, we can filter by the tensor length \(n\) minus the number \(k\) of tensors not equal to one, that is to say, \(k\) is the cardinality of a subset \(I\) in
\[
  \bigl(\mathsf{sgn}_n\otimes \Lie(n)\bigr)\otimes_{\bS_n}\!\bigl(\Com(n/I)\otimes \mathsf{coCom}^{\sconv I}(m)\bigr).
\]
This way we first consider the cohomology of reduced Harrison complexes 
\[
  \prod_{k\geq 1} \Lie(k) \otimes_{\bS_k} \!\bigl(\susp{-1}\Q[x_1,\dots,x_m]_+ \bigr)^{\otimes k}
\]
which, of course, is the space of generators, \(\Q x_1\oplus\dots\oplus\Q x_m\). Adding the condition on homogeneous multidegree, we are led to conclude that \(C(\Com, m)\) is acyclic if \(m\neq 1\), and if \(m=1\) it is quasi-isomorphic to the subcomplex spanned by words consisting of any number of \(1\)'s, and a single \(x_1\). This subcomplex in turn is easily seen to have one-dimensional cohomology, given by
\[
  (\,1\mid x_1\,) + (\,x_1\mid 1\,).
\]
Thus: The left action of the generator \(m_2\) of \(\Com\) defines a quasi-isomorphism
\[
  m_2\circ_1(-) + m_2\circ_2(-) : U(1) \to C\bigl(\Com, \Com\infcirc U^\Com\bigr)
\]
from \(U(1)\) to the Harrison cohomology of the infinitesimal bimodule freely generated by the cofree right module on \(U\).

\subsubsection{Derived functor interpretation}

The Harrison cohomology
\[
  C\bigl(\Com, \Com\infcirc(-)\bigr): \Mod{\Com} \to \dgV
\]
represents the total right derived functor of the functor that takes a module \(M\) to the kernel of the right action \(M(1)\otimes \Com(2)\to M(2)\). Indeed, if \(M=U^{\Com}\), then this kernel is \(U(1)\) and, as we have seen, we have a monomorphic quasi-isomorphism \(U(1)\to C(\Com,\Com\infcirc M)\). Every right module admits an injective resolution by cofree modules, and the statement follows.

\subsubsection{A monoidal property}

We shall here prove here that \(\Com\infcirc(-)^\Com\) is strong monoidal for the Day convolution and the arity-wise tensor product. First note
\[
  (U\Conv V)^{P}(n) \cong \int_{I,\hspace{.4pt}J}\Map\Bigl(P^{\sconv n}(I+J), U(I)\otimes V({\mkern 2mu}J) \Bigr).
\]
If \(P\) is a unitary operad, meaning \(P(0)=\Q\), which by operad composition gives distinguished ``projections'' \(P(k)\to P(k-1)\), and additionally is a Hopf operad, meaning it has a diagonal morphism \(P\to P\otimes P\) to the arity-wise tensor product with itself, then there will be morphisms
\begin{equation}\label{eqn:monoid}
  P^{\sconv n}(I+J) \to \bigoplus_{\substack{ K,L\subset [n] \\ [n]=K\cup L}} 
  P^{\sconv K}(I)\otimes P^{\sconv L}({\mkern 2mu}J).
\end{equation}
Namely, \(P^{\sconv n}(I+J)\) is a direct sum over partitions of \(I+J\) into \(n\) parts, which is to say, over surjections \(f:I+J\to [n]\). Given such a surjection, let \(I_i = I\cap f^{-1}(i)\) and \(J_j = J\cap f^{-1}({\mkern 1.5mu}j)\). Then take \(K = \{i \mid I_i\neq\emptyset\}\) and \(L=\{{\mkern 1mu}j\mid J_j\neq\emptyset\}\). Evident use of the Hopf diagonal and the projections then maps into component \(P^{\sconv K}(I)\otimes P^{\sconv L}({\mkern 2mu}J)\).

The operad \(\Com\) has an extension to a such a unitary Hopf operad, the operad \(\mathsf{uCom}\) of commutative algebras with a unit. In this way we get a morphism as in equation \ref{eqn:monoid}.
Since \(\Com(n) = \Q\) for all \(n\), we deduce by comparing the index sets that it is an isomorphism. (The map can be written down in a more direct way, but we thought it worth-wile to indicate that such a morphism is not unique to the commutative operad.) This directly implies that
\[
 (U\Conv V)^{\Com}(n) = \bigoplus_{\substack{ K,L\subset [n] \\ [n]=K\cup L}} U^{\Com}(K)\otimes U^{\Com}(L).
\]
Then, \(\Com\infcirc(-)^\Com\) is monoidal because
\begin{gather*}
 \bigoplus_{\substack{K\subset [n]}} \Com(n/K)\otimes U^{\Com}(K)\otimes \bigoplus_{\substack{L\subset [n]}} \Com(n/L) \otimes V^{\Com}(L) \\
 = \bigoplus_{M\subset [n]} \Com(n/M)\otimes 
       \bigoplus_{\substack{ K,L\subset M \\ M=K\cup L}} U^{\Com}(K)\otimes U^{\Com}(L)
\end{gather*}

\subsubsection{Vanishing of Harrison cohomology}

The commutative operad has a canonical morphism \(\Com \to \Com\otimes\Com\) (it is a Hopf operad), so the arity-wise tensor product \(M\otimes N\) of two infinitesimal \(\Com\)-bimodules is again an infinitesimal bimodule. It follows from the two preceding sections that if \(M=\Com\infcirc U^{\Com}\) and \(N=\Com\infcirc V^{\Com}\), then \(C(\Com, M\otimes N)\) is acyclic, because \(M\otimes N\) is isomorphic to \(\Com\infcirc (U\Conv V)^{\Com}\) and hence
\[
  C(\Com, M\otimes N) \simeq (U\Conv V)(1) = 0,
\]
since we allow no collections to have a component in arity zero. Since every right module has an injective resolution by cofree modules, it actually holds for all \(M\) and \(N\) that are freely generated by right \(\Com\)-modules.

\subsubsection{Bimodules with free left action}

If an infinitesmial \(P\)-bimodule \(M\) has a free infinitesimal left action, then \(M=P\infcirc(\mathbb{I}\infcirc_P \!M)\) when considered as just a left (infinitesimal) module, but the equality need not hold as a bimodule.

If for example \(P=\Com\) and \(M\) has an extension to a module for the unitary commutative operad \(\mathsf{uCom}\) (which has \(\mathsf{uCom}(0)=\Q\), that act on a module as ``projections'' \(M(n)\to M(n-1)\), akin to codegeneracies of a cosimplicial set), then \(M\) automatically has a free left action because in this case all left actions must be monomorphisms.\footnote{\;Compare with cosimplicial sets: the codegeneracies enforce that coface maps must be injective.} We have excluded arity zero and hence unitary operads, but many modules considered elsewhere in the paper have such an extension implicit and, thus, have free left actions. That an infinitesimal bimodule is not freely generated by a right module simply because the left action is free is a technical nuisance, that we shall need to work around.\footnote{\;Tamarkin claims that \(\Bimod{\mathsf{uCom}}\) and \(\Mod{\Com}\), called by him \(\textbf{Shpart}\) and \(\textbf{Shsur}\), are equivalent categories, failing to make this distinction between free left action and freely generated by a right module.}

If \(M\) is an infinitesimal \(\Com\)-bimodule with a free left action, then as a graded vector space
\[
  C(\Com, M) = \prod_{\substack{n\geq 2\\ I\subset n}} s\Sigma\,\Lie(n) \otimes_{\bS_n} \Com(n/I)\otimes N(I)
\]
with \(N=\mathbb{I}\infcirc_{\Com} M\). Filter this by \(n-\#I\). The associated graded complex is equal to \(C(\Com, \Com\infcirc N)\). We may assume \(N=U^{\Com}\) because the category of right modules has enough injectives of the form \(U^{\Com}\). The cohomology only depends on \(U(1)\), as we have seen. It follows that the cohomology of \(M\otimes M\) vanishes because the spectral sequence then depends on \((U\Conv U)(1)\).

\begin{lem}\label{lem:tamarkin}
If\, \(M\) and \(N\) are infinitesimal \(\Com\)-bimodules with free left actions, then the Harrison cohomology of\, \(M\otimes N\) is acyclic.
\end{lem}

%\printbibliography
\bibliographystyle{plain}
\bibliography{dihgrt}

\end{document}